\documentclass{amsart}

\usepackage{amssymb,
amsmath,
amsthm,
graphicx,
amscd,
multind,
eufrak
}
\usepackage{verbatim}

\newcommand{\cha}{\operatorname{char}}
\theoremstyle{plain}
\newtheorem{thm}{Theorem}[section]
\newtheorem{lm}[thm]{Lemma}

\newtheorem{prop}[thm]{Proposition}

\newtheorem{lem}[thm]{Lemma}
\theoremstyle{definition}
\newtheorem{de}[thm]{Definition}

\newtheorem{ex}[thm]{Example}
\newtheorem{re}[thm]{Remark}

\newcommand{\CC}{{\mathbb C}}
\newcommand{\RR}{{\mathbb R}}
\newcommand{\ZZ}{{\mathbb Z}}
\newcommand{\NN}{{\ZZ_{\geq 0}}}
\newcommand{\NNp}{{\ZZ_{> 0}}}

\newcommand{\im}{\operatorname{im}}
\newcommand{\supp}{\operatorname{supp}}

{\begin{figure} \begin{center}}%
{\end{center} \end{figure}}

\newcommand{\id}{\operatorname{id}}
\newcommand{\rk}{\operatorname{rk}}
\newcommand{\bb}{\mathbf{b}}

\newcommand{\Inc}{\operatorname{Inc}}

\newcommand{\GL}{\operatorname{GL}\nolimits}

\newcommand{\G}{{H}}            
\newcommand{\g}{{h}}            

\newcommand{\Gr}{\operatorname{Gr}}
\newcommand{\coord}{\mathcal{O}}      

\newcommand{\n}{{n_0}}          
\newcommand{\m}{{m_0}}          
\newcommand{\p}{{p}}            
\newcommand{\mm}{{m}}           
\newcommand{\ms}{{\mu}}         
\newcommand{\pol}{{f}}
\newcommand{\M}{M}              
\newcommand{\ten}{{\xi}}        
\newcommand{\con}{{\xi}}        
\newcommand{\vect}{{v}}
\newcommand{\tc}{{\zeta}}       

\newcommand{\wo}{{u}}           
\newcommand{\ve}{{q}}           
\newcommand{\vet}{{\ve'}}

\newcommand{\ro}{{r}}            
\newcommand{\lea}{{\operatorname{leaf}}}        
\newcommand{\Int}{{\operatorname{int}}}         
\newcommand{\ver}{{\operatorname{vert}}}        
\newcommand{\edge}{\operatorname{edge}}
\newcommand{\rep}{{\operatorname{rep}_G}}       
\newcommand{\eqm}{{\mathrm{CV}}}                
\newcommand{\eqminf}{{\eqm_{\infty}}}           
\newcommand{\eqmo}{{\Psi}}                      
\newcommand{\tre}{{T}}          
\newcommand{\di}{{d}}           
\newcommand{\degr}{{D}}         
\newcommand{\Rhom}{{\Phi}}      
\newcommand{\pull}{{\phi}}      

\newcommand{\Ainf}{A_\infty}

\newcommand{\Cinf}{\coord_\infty}
\newcommand{\Ginf}{\G_\infty}

\newcommand{\Sinf}{S_\infty}

\newcommand{\Zinf}{{Z_\infty}}

\newcommand{\Yinf}[1][k]{Y_{\infty}^{\leq #1}}
\newcommand{\Yp}[2][k]{Y_{#2}^{\leq #1}}
\newcommand{\bw}{\mathbf{w}}

\newcommand{\bwo}{\mathbf{\wo}}

\begin{document}
\title{Finiteness results for Abelian tree models}

\begin{abstract}
Equivariant tree models are statistical models used in the
reconstruction of phylogenetic trees from genetic data.
Here {\em equivariant} refers to a symmetry group
imposed on the root distribution and on the transition matrices
in the model. We prove that if that symmetry group is Abelian,
then the Zariski closures of these models
are defined by polynomial equations of bounded degree, independent of the tree.
Moreover, we show that there exists a polynomial-time membership
test for that Zariski closure. This generalises earlier
results on tensors of bounded rank, which correspond to the
case where the group is trivial and the tree is a star, and implies a qualitative variant
of a quantitative conjecture by Sturmfels and Sullivant in
the case where the group and the alphabet coincide. Our
proofs exploit the symmetries of an infinite-dimensional projective limit of
Abelian star models.
\end{abstract}

\author[J.~Draisma]{Jan Draisma}
\address[Jan Draisma]{
Department of Mathematics and Computer Science\\
Technische Universiteit Eindhoven\\
P.O. Box 513, 5600 MB Eindhoven, The Netherlands;
and VU and CWI, Amsterdam, The Netherlands}
\thanks{Both authors are supported by the first author's Vidi grant from
the Netherlands Organisation for Scientific Research (NWO)}
\email{j.draisma@tue.nl}

\author[R.H.~Eggermont]{Rob H. Eggermont}
\address[Rob H. Eggermont]{
Department of Mathematics and Computer Science\\
Technische Universiteit Eindhoven\\
P.O. Box 513, 5600 MB Eindhoven, The Netherlands}
\email{r.h.eggermont@tue.nl}

\maketitle

\tableofcontents

\section{Introduction}

Tree models are families of probability distributions used in modelling
the evolution of a number of extant species from a common ancestor. Here
{\em species} can refer to actual biological species, but tree models have
also been applied to other forms of evolution, e.g. of languages. The hypothesis
underlying tree models is that DNA-sequences of those extant species,
arranged and suitably aligned in a table with one row for each species,
can be meaningfully read off column-wise. Indeed, these columns (or
{\em sites}) are assumed to be independent draws from one and the same
probability distribution belonging to the model.

To describe that model, one fixes a finite rooted tree $\tre$ whose leaves
correspond to the species and whose root $\ro$ corresponds to the common
ancestor. One also fixes a finite alphabet $B$. The case
where $B=\{A,C,G,T\}$ is the alphabet of nucleotides is of
most interest in biology, but the theory developed here works
for arbitrary finite $B$. Associated to each vertex of the
tree is a copy of $B$. To $\ro$ one attaches a probability distribution
$\pi$ on $B$, and to each edge $\ve \to \vet$, directed away from $\ro$,
one attaches a $B \times B$-matrix $A_{\ve\vet}$ of real non-negative numbers
whose row sums equal $1$. Its entry $A_{\ve\vet}(b,b')$ at position $(b,b')$
records the probability that the letter $b$ at vertex $\ve$ mutates into the
letter $b'$ at vertex $\vet$. The random process modelling evolution of the
nucleotide at a single position consists of drawing a letter $b \in B$
from the distribution $\pi$ and mutating it along the edges with the
probabilities given by the matrices $A_{\ve\vet}$. The probability that
this leads to a given word $\bb \in B^{\lea(\tre)}$ equals
\[ P(\bb)=\sum_{\bb' \in B^{\ver(\tre)} \text{ extending }
\bb} \pi(b'_\ro) \cdot \prod_{\ve \to \vet \in \edge(\tre)} A_{\ve\vet}(b'_\ve,b'_\vet), \]
Now as the root distribution $\pi$ and the transition matrices
$A_{\ve\vet}$ vary, the set of all probability distributions $P \in
\RR^{(B^{\lea(\tre)})}$ thus obtained is called the model.  The fact
that the entries of $P$ are polynomial functions of the parameters
has led to an extensive study of the {\em algebraic variety} swept
out by this parameterisation, by which we mean the Zariski closure in
$\RR^{(B^{\lea(\tre)})}$ (or even $\CC^{(B^{\lea(\tre)})}$) of the model
\cite[Chapter 4]{Pachter05}; see also the expository paper
\cite{Cipra07}. The present paper also concerns that Zariski closure.

The model without further restrictions on the root distributions $\pi$
or the transition matrices $A_{\ve\vet}$ is known as the {\em general
Markov model} for the tree $\tre$ and the alphabet $B$. In applications
the number of parameters is often reduced by impo\-sing further symmetry,
reflecting additional biological (or, say, linguistic) structure. This
is often\footnote{But not always! Most notably, the general {\em
time-reversible Markov model}, where the only restriction on the
transition matrices is that they be symmetric, is not of this form for
$|B|>2$. We have not tried to generalise our results to this case.}
done by choosing a finite group $G$ acting by permutations on the set
$B$, requiring that $\pi$ be a $G$-invariant distribution (which when
$G$ acts transitively means that it is the uniform distribution),
and requiring that each transition matrix $A_{\ve\vet}$ satisfies
$A_{\ve\vet}(gb,gb')=A_{\ve\vet}(b,b')$ for all letters $b,b' \in B$. The
resulting model, which is a subset of $\RR^{(B^{\lea(\tre)})}$ contained
in the general Markov model, has been dubbed the {\em equivariant tree
model} for the triple $(\tre,B,G)$ \cite{Draisma07b}; here we implicitly
mean that the action of $G$ on $B$ is also fixed. The special case where
$G$ is Abelian and $B=G$ with the left action of $G$ on itself is called a
{\em group-based model}. Our first two main theorems concern the class of
equivariant tree models for which $G$ is Abelian, but does not necessarily
act transitively on $B$.  This class includes the general Markov model
(with $G=\{1\}$) as well as group-based models.

\begin{thm}[Main Theorem I]
For any action of an Abelian group $G$ on a finite alphabet $B$, there
exists a uniform bound $\degr=\degr(B,G)$ such that for any finite tree
$\tre$ the Zariski closure of the equivariant tree model for $(\tre,B,G)$
is defined by polynomial equations of degree at most $\degr$.
\end{thm}

In fact, we will prove the stronger statement that {\em finitely many
types of equations} suffice to define the Zariski closures of the
equivariant tree models for all $\tre$. For the general Markov model, this result first appeared in
\cite{Draisma11d}. For group-based models, where the Zariski closure of
(the cone over) the tree model for $(\tre,B,G)$ is a toric variety, a
much stronger conjecture was put forward in \cite{Sturmfels05b}, namely,
that for any tree $\tre$ the ideal of that toric variety is generated by
binomials of degree at most $|G|$. This would imply that $\degr(B,G)=|G|$
suffices when $G$ acts transitively on $B$. Our result is weaker in
that we do {\em not} prove the existence of a degree bound for polynomials
generating the ideal---our result is {\em set-theoretic} rather than
{\em ideal-theoretic}---and that we do not find an explicit bound.
Nevertheless, Main Theorem I is the first general finiteness result even
for the restricted class of group-based models, though for group-based
models more recent work by Michalek \cite{Michalek12} gives finiteness
results at the level of projective schemes, which are somewhere between
set-theoretic and ideal-theoretic results.

\begin{thm}[Main Theorem II]
For any action of a finite group $G$ on a finite alphabet
$B$, there exists a polynomial-time algorithm that, on input
a tree $\tre$ and a probability distribution $P$ on $B^{\lea(\tre)}$
determines if $P$ lies in the Zariski closure of the equivariant tree
model for $(\tre,B,G)$.
\end{thm}

We hasten to say that our proofs are non-constructive. In particular,
they do not yield an explicit bound $\degr(B,G)$ and they do not
give an explicit algorithm---though the overall structure of that
algorithm is clear, see Section~\ref{sec:Proofs}. This situation is
reminiscent of Robertson-Seymour's non-constructive proof that any
minor-closed property of finite graphs can be tested in polynomial time
\cite{Robertson95,Robertson04}. In Main Theorem II, the notion of {\em
polynomial-time algorithm} depends on the (machine) representation of the
entries of $P$. If they are rational numbers, then we mean polynomial-time
in the bit-size of $P$ (in a non-sparse representation, i.e., zero entries
count). If they are abstract real numbers, then we mean a Blum-Shub-Smale
machine \cite{Blum89} whose number of arithmetic operations on real
numbers is bounded by some polynomial in $|B|^{|\lea(\tre)|}$.

Our Main Theorems I and II do {\em not} require that the trees $\tre$ be
trivalent. Indeed, for the class of trivalent trees, or indeed for the
class of trees with any fixed upper bound on the valency of internal
vertices, Main Theorems I and II are relatively easy consequences
of known results from
\cite{Allman04,Casanellas05,Sturmfels05b,Draisma07b}, which express the ideal of equations of an
equivariant tree model in terms of ideals of equivariant tree models
of {\em star trees}. Bounding the degree of polynomial equations for
large star models and the complexity of testing membership of their
Zariski closures is the real challenge in this paper. We stress that
this leaves open the question of actually finding (practical) algorithms
for testing membership of (Zariski closures of) tree models.
Our results should be interpreted as a theoretical contribution to the
algebraic statistics of tree models.

However, we do believe that some of the techniques that go into the
proofs of our Main Theorems I and II can be of practical use. In particular,
one crucial observation in our proofs is the following.  Consider the
equivariant star model for the triple $(\tre,G,B)$, where $\tre$ is a
star and where $G$ needs not be Abelian. Label the leaves of $\tre$ with
$0,\ldots,\mm-1$, so that $B^{\lea(\tre)}$ can be identified with $B^\mm$.
Fix a natural number $\n \leq \mm$ and any probability distribution $Q$
on $B^{\n}$ that is invariant with respect to the diagonal $G$-action
on $B^{\n}$. Then for any probability distribution $P$ on
$B^\mm$ we can define a probability distribution $P_Q$ on $B^{\mm-\n}$ by
\[ P_Q(\bb)=\frac{\sum_{\bb' \in B^{\n}} P(\bb,\bb') Q(\bb')}{Z} \]
where $P(\bb,\bb')$ is the probability of observing $\bb$ at positions
$0,\ldots,\mm-\n-1$ and $\bb'$ at positions
$\mm-\n,\ldots,\mm-1$. Here $Z$ is
a normalising factor, and a condition for this to be well-defined is
that $Z$ is non-zero. Let $\tre'$ be the tree obtained from
$\tre$ by deleting the last $\n$ leaves. Our elementary but useful
observation is that, for any fixed $G$-invariant $Q$, the (partially defined)
map $P \mapsto P_Q$ maps the equivariant model for $(\tre,G,B)$ into
the equivariant model for $(\tre',G,B)$. As a consequence, equations
for the latter model pull back to equations for the former model, and a
necessary condition for $P$ to be in
(the Zariski closure of) the former model is that for all $G$-invariant
$Q$ the distribution $P_Q$ lies in the latter model.

In the course of proving Main Theorems I and II we show that for some
suitable $\n$, chosen after fixing $G$ and its action on $B$, and for
some suitably chosen set of $G$-invariant probability distributions $Q$
on $B^{\n}$, the converse also holds: if a probability distribution $P$
on $B^\mm$ with $\mm \gg \n$ has the property that $P_Q$ lies in the star
model with $\mm-\n$ leaves for all chosen $Q$ on all cardinality-$\n$
subsets of the leaves, then $P$ lies in the star model with $\mm$
leaves. We do this by constructing an infinite-dimensional limit of all
$\mm$-star models for the pair $(G,B)$---or rather $\n$ of these limits,
one for each congruence class of $\mm$ modulo $\n$---and showing that
this limit lies in some infinite-dimensional {\em flattening variety}
that is Noetherian up its natural symmetries. This is also the technique
followed in \cite{Draisma11d} for the case where $G=\{1\}$; there $\n$
can be taken $1$. We simplify some of the arguments from that paper,
but our present, more general results are more subtle since they really
require the use of jumps by some carefully chosen $\n>1$.

This paper is organised as follows. In Section~\ref{sec:mainthm}
we briefly recall the well-known tensorification of the set-up above
(see, e.g. \cite{Allman04,Draisma07b}) and state two theorems for this
setting. Then in Section~\ref{sec:Tensors} we give some properties of
tensors in finite-dimensional $G$-representations that will motivate
the use of flattenings and our choice for $\n$.

In Section~\ref{sec:Infinite}, after fixing any value for $\n$, we introduce an
infinite-dimensional ambient space (again, $\n$ of these, one for each
congruence class modulo $\n$), containing an infinite-dimensional limit
of the equivariant models for finite stars; we dub this {\em the infinite
star model}. In this section we define the {\em flattening variety} as well,
a variety containing the infinite star model. This variety is defined
by determinantal equations of bounded degree, roughly corresponding to
the coarser star models where the leaves of a tree are partitioned into
two subsets.  We prove that the flattening variety is defined by finitely
many orbits of determinantal equations under the natural symmetry group
of the infinite tree model. Then in Section~\ref{sec:EqNoeth} we prove that
the flattening variety is Noetherian under this symmetry group. Finally,
our main theorems are derived from this in
Section~\ref{sec:Proofs}, and it is only here that we
need the infinite star model mentioned before.

We conclude this introduction with a list recording values
of our uniform bound $\degr(B,G)$ that are known to us.
\begin{description}
\item[Binary general Markov model] Here $G=\{1\}$ and $B$
has cardinality two, and results from \cite{Landsberg04} imply that
$\degr(B,G)$ can be taken equal to $3$; apart from linear
equations expressing that probabilities sum up to $1$, the
degree-$3$ equations are
the determinantal equations defining the flattening variety
(see Section~\ref{sec:Infinite}). The paper \cite{Raicu10}
proves the stronger statement, previously known as the
GSS-conjecture \cite{Garcia05}, that these equations generate the ideal of (the cone over)
the general Markov model.

\item[Binary Jukes-Cantor model] This is the group-based
model with $G=B=\ZZ/2\ZZ$, and results from \cite{Sturmfels05b} show that
$\degr(B,G)$ can be taken equal to $2$. The non-linear,
quadratic equations are determinantal equations defining the
finer flattening variety $\Yp[(k_\chi)_\chi]{[\mm]}$ from
Remark~\ref{re:Remarks}, item 5, and these generate the
ideal of the cone over the model. The algebra and geometry of this
model for varying trees is further studied in
\cite{Buczynska07}.

\item[Kimura 3-parameter model] This is the group-based
model with $G=B=\ZZ/2\ZZ \times
\ZZ/2\ZZ$, and results from
\cite{Michalek12}
show that $\degr(B,G)$ can be taken equal to $4$. The
degree-$4$
equations were known from \cite{Sturmfels05b}, where it was
conjectured that they generate the ideal. The result of
\cite{Michalek12} is slightly weaker than that but stronger than the
purely set-theoretic statements that we are after. The
geometry of this model is also studied in
\cite{Casanellas07,Michalek11}.
\end{description}

If one restricts oneself to trivalent trees, then
more is known for other models, as well, such as the
strand-symmetric model \cite{Casanellas05} or the
all-important 4-state general
Markov model \cite{Allman04,Friedland11,Bates10} or
further group-based models with small groups $G$
\cite{Sturmfels05b}.

One might wonder where the restriction to Abelian $G$ comes from; after
all, tree models for which $G$ is not Abelian are used in practice.
At this point, before going through the proofs, all we can say is that
they break down at the point where we prove that the infinite-dimensional
flattening variety is defined by finitely many orbits of equations;
see also Remark~\ref{re:NonAbelian}.

Finally, a word of self-criticism is in order here: it is unclear whether
the degree bound and the algorithm from our main theorems will be useful
in phylogenetic practice, even if they are made explicit. In phylogenetic
reconstruction, certain determinantal equations coming from edges often
suffice to distinguish the model for one tree from the model for another
tree (with $G$ and $B$ fixed) \cite{Casanellas09}. On the other hand,
our characterisation of general Abelian tree models using contractions
and flattenings gives more insight into the geometry of these models,
and our infinite-dimensional methods will likely apply to other models
from algebraic statistics.

\section*{Acknowledgments}
We thank Bernd Sturmfels and an anonymous referee for many helpful
suggestions for improving the text.

\section{Tensor formulation of the main results} \label{sec:mainthm}
Before we recall the tensorification of the model mentioned in the
introduction, we introduce notation that will be used throughout this
article.  Let $G$ be a finite Abelian group. For us, a $G$-representation
over a field $K$ will be assumed to be finite-dimensional, unless
explicitly mentioned otherwise.  Let $K$ be an infinite field such
that every $G$-representation over $K$ splits into
a direct sum of one-dimensional representations. For this it suffices,
for instance, that $K$ is algebraically closed and that $\cha K$ does not
divide $|G|$.  For $\mm \in \NN$, set $[\mm]:=\{0,\ldots,\mm-1\}$. If
$V_i$ is a $G$-representation over $K$ for each $i\in [\mm]$ and if
$I \subseteq [\mm]$, then we write $V_I:=\bigotimes_{i \in I} V_i$ for the
tensor product of the $V_i$ with $i \in I$. The rank of a tensor $\omega$
in $V_{I}$ is the minimal number of terms in any expression of $\omega$
as a sum of pure tensors $\bigotimes_{i \in I} \vect_i$ with $\vect_i
\in V_i$. A tensor $\omega$ has \emph{border rank} at most $k$ if it
lies in the Zariski closure of the set of tensors of rank at most $k$.

Given an $\mm$-tuple of linear maps $\phi_i: V_i \to U_i$,
where $U_i$ is also a vector space over $K$ for each $i \in [\mm]$, we write
$\phi_{[\mm]}:=\bigotimes_{i \in [\mm]}\phi_i$ for the linear map $V_{[\mm]} \to U_{[\mm]}$ determined by $\bigotimes_{i \in [\mm]} \vect_i \mapsto \bigotimes_{i
\in [\mm]} \phi_i(\vect_i)$. Clearly $\rk \phi_{[\mm]} \omega \leq \rk \omega$ for
any $\omega \in V_{[\mm]}$, and this inequality carries over to the border
rank.

If $I \subseteq [\mm]$ and $\ten \in \bigotimes_{i \in
I}V_i^*$, then the tensor $\ten$ induces a linear map
$V_{[\mm]} \rightarrow V_{[\mm] - I}$. We call this map the
\emph{contraction} along the tensor $\ten$; except for a
normalising factor, it is the
tensorial analogue of the map $P \mapsto P_Q$ from the
introduction. This map is $G$-equivariant if and only if $\ten$ is $G$-invariant; moreover, it does not increase the rank or the border rank of any element of $V_{[\mm]}$. We can now state our third main theorem.

\begin{thm}[Main Theorem III]
For all $k \in \NN$ there exists $\M$ such that for all $\mm > \M$
and for all $G$-modules $V_i$ over $K$ with $i \in [\mm]$, a tensor $\omega \in V_{[\mm]}$ has border rank at most $k$ if and only if for all $\ms \leq \M$, all its contractions in $\mm-\ms$ factors along $G$-invariant tensors have border rank at most $k$.
\end{thm}

The novelty in this theorem, compared to the results in \cite{Draisma11d},
is that it suffices to contract along $G$-invariant tensors rather than
general tensors, at the cost of increasing the dimension of those tensors
to be contracted with. While not strictly necessary for our other main
results, Main Theorem III illustrates the general approach taken in this
paper, which is to replace ``baby steps'' for $G=1$ with ``giant steps''
for general Abelian $G$.  Our fourth main theorem, which generalises
our first main theorem, requires a bit more work to formulate.

\begin{de}\label{def:Gspaced}
A \emph{$G$-spaced tree} is a tree $\tre$ together with for each vertex $\ve$ a $G$-module $V_\ve$, a distinguished basis $B_\ve$ of $V_\ve$ such that $G$ acts on $B_\ve$ and a non-degenerate symmetric bilinear form $(. | .)_\ve$ defined by the property that $B_\ve$ is an orthonormal basis with respect to $(. | .)$. For vertices $\ve,\vet$, we say $\ve \sim \vet$ if and only if $(\ve,\vet)$ is an edge of $\tre$. We denote by $\ver(\tre)$, $\Int(\tre)$, respectively $\lea(\tre)$, the set of vertices, internal vertices, respectively leaves, of $\tre$. We define
\[ L(\tre) := \bigotimes_{\ve \in \lea(\tre)}V_\ve\quad \text{
and } \quad R(\tre) := \bigotimes_{\ve \in
\ver(\tre)}V_\ve^{\otimes \{\vet \sim \ve\}}. \]
Let $\tre$ be a $G$-spaced tree. A \emph{$G$-representation} of $\tre$ is a collection $(A_{\vet\ve})_{\vet \sim \ve}$ of $G$-invariant elements of $V_\vet \otimes V_\ve$ such that for any $\vet \sim \ve$, the tensor $A_{\vet\ve}$ maps to $A_{\ve\vet}$ via the natural isomorphism $V_\vet \otimes V_\ve \to V_\ve \otimes V_\vet$. The set of $G$-representations of $\tre$ is denoted $\rep(\tre)$.
\end{de}

Note that in the set-up of the introduction, each vertex of
the tree has the same space attached; in other words, there
is some $G$-representation $V$ with some fixed basis $B$,
some fixed symmetric bilinear form $(. | . )$ (and some
fixed action of $G$) such that $V_\ve = V$, $B_\ve = B$ and
$(. | . )_\ve = (. | .)$ for any vertex $\ve$ of the tree.
In this setting, we can
view a probability distribution $P \in \CC^{B^{\lea(\tre)}}$
as an element of $L(\tre)$; namely, we can identify $P$ with
\[ \sum_{\bb = (b_\ve)_{\ve \in \lea(\tre)} \in
B^{\lea(\tre)}}P(\bb)\cdot \bigotimes_{\ve \in
\lea(\tre)}b_\ve. \] This is the tensorification of the set-up
of the introduction. For our purposes, we will need to use
the more flexible setting of Definition~\ref{def:Gspaced},
as we will want to apply theorems proved in
\cite{Draisma07b}. Usually however, it will suffice to
consider trees for which each vertex has the same space
attached; see for example Lemma~\ref{lm:LeafSame}.

There is a canonical isomorphism $\rep(\tre) \to R(\tre)$,
defined by the embedding of elements in the tensor product
of the $V_\vet \otimes V_\ve$ ranging over the unordened
pairs of edges $\{\vet \sim \ve, \ve \sim \vet\}$  into
$R(\tre)$. We denote by $\Psi$ (or sometimes $\Psi_{\tre}$
to indicate which tree we are talking about) the composition
of this map with the contraction $R(\tre) \to L(\tre)$ along
the ($G$-invariant) tensor $\bigotimes_{\vet \in \Int(\tre)}\sum_{b \in B_\vet}(b\mid . \hspace{0.1 cm})^{\otimes \{\ve \sim \vet\}}$.

\begin{de} The \emph{equivarant model} $\eqm(\tre)$ associated to a tree
$\tre$ is the Zariski closure $\overline{\Psi(\rep(\tre))}$
of the image of $\Psi$.
\end{de}

Note the slight discrepancy with the introduction, where the term equivariant model was
used for the image of $\Psi$ on stochastically meaningful
parameters. But the present definition is the one used in
\cite{Draisma07b}, from which we will use some results.
While there the group $G$ was allowed to be arbitrary, we
stress once again that in the present paper we only consider
{\em Abelian} $G$.
We can now state our fourth main theorem.

\begin{thm}[Main Theorem IV] If $K$ is algebraically closed and of characteristic zero, then for all $k \in \NN$, there exists a $\degr \in \NN$ such that for each $G$-spaced tree $\tre$ such that $|B_{\ve}| \leq k$ for each $\ve \in \Int(\tre)$, the variety $\eqm(\tre)$ is defined by the vanishing of a number of polynomials of degree at most $\degr$.
\end{thm}

The bound $\degr$ will certainly have to depend on $k$. For instance, if
$G$ is the trivial group, and $\tre$ is a star tree, then the variety
$\eqm(\tre)$ is the variety of tensors of rank at most $k$, and no
polynomials of degree less than $k+1$ vanish on this variety.
Main Theorem I is a direct corollary of this theorem; the
details for passing from the case of unrooted
trees without the restriction that row sums of transition
matrices are $1$ to the case of rooted trees with that
additional restriction can be found in Section 3 of
\cite{Draisma07b}.

\section{Tensors and flattening} \label{sec:Tensors}
In the proofs of our main theorems, in addition to
contractions, we will use a second operation
on tensors, namely, {\em flattening}. Suppose that $I,J$ form a
partition of $[\mm]$ into two parts. Then there is a natural isomorphism
$\flat=\flat_{I,J}: V_{[\mm]} \to V_I \otimes V_J$. The image $\flat
\omega$ is a $2$-tensor called a flattening of $\omega$. Its rank (as
a $2$-tensor) is a lower bound on the border rank of $\omega$.
The first step in our proof below is a reduction to the case where all
$V_i$ are isomorphic as $G$-representations. Here, $i$ can either be viewed as an element of $[\mm]$ (in Main Theorem III) or as an element of $\lea(\tre)$ (in Main Theorem IV).

We have the following lemma, in which $K[G]$ stands for the regular representation of $G$.

\begin{lm} \label{lm:AllSame}
Let $\mm,k,n$ be natural numbers with $n \geq k+1$, and let $V_0,\ldots,V_{\mm-1}$
be $G$-representations over $K$. Then a tensor $\omega \in V_{[\mm]}$ has rank (respectively, border rank) at most $k$ if and only
if for all $\mm$-tuples of $G$-linear maps $\phi_i:V_i \to K[G]^n$ the tensor
$\phi_{[\mm]}(\omega)$ has rank (respectively, border
rank) at most $k$.

Moreover, if $\omega \in V_{[\mm]}$ has border rank at most
$k$, then there exist $G$-linear maps $\phi_i: V_i
\to K[G]^k$ and $\psi_i: K[G]^k \to V_i \
(i=1,\ldots,\mm)$,
such that $\psi_{[\mm]}(\phi_{[\mm]}(\omega)) = \omega$.
\end{lm}

This lemma holds at the scheme-theoretical level, but we will not need
that. For $G$ the trivial group, the lemma reduces to \cite[Theorem
11]{Allman04}.

\begin{proof}
The ``only if'' part follows from the fact that $\phi_{[\mm]}$ does
not increase rank or border rank. For the ``if'' part assume that
$\omega$ has rank strictly larger than $k$, and we argue that there
exist $\phi_0,\ldots,\phi_{\mm-1}$ such that $\phi_{[\mm]}(\omega)$ still has
rank larger than $k$. It suffices to show how to find $\phi_0$; the
remaining $\phi_i$ are found in the same manner. Let $U_0$ be the image
of $\omega$ regarded as a linear map from the dual space $V_{[\mm] - \{0\}}^*$
to $V_0$. Set
\[ U_0' := K[G]U_0 = \left\{\sum_{g \in G}c_ggu: c_g
\in K, u \in U_0\right\}. \]
For each irreducible $G$-representation $\chi$, let $k_{\chi}$ be the multiplicity of $\chi$ in $U_0'$. If $k_{\chi}$
is at most $n$ for each $\chi$, then by elementary linear
algebra and the fact that $K[G]$ is the sum of all irreducible representations of $G$ there exist $G$-linear maps $\phi_0:V_0 \to K[G]^n$ and $\psi_0:K[G]^n \to V_0$ such that $\psi_0 \circ \phi_0$ is the identity map on $U_0'$, and hence on $U_0$. Set
$\omega':= (\phi_0 \otimes (\bigotimes_{i>0} \id_{V_i}))(\omega)$, so that by construction $\omega$ itself equals $(\psi_0 \otimes (\bigotimes_{i>0}
\id_{V_i}))(\omega')$. By the discussion above, we have the inequalities
$\rk \omega \geq \rk \omega' \geq \rk \omega$, so that both ranks are
equal and larger than $k$, and we are done. If, on the other hand, there is $\chi$ such that $k_{\chi} > n$, then let $\phi_0:V_0 \to K[G]^n$ be any
$G$-linear map that maps the $\chi$-component of $U_0'$ surjectively onto the $\chi$-component of $K[G]^n$ for each $\chi$ with $k_{\chi} > n$. Then the image of $U_0$ must have rank at least $n$. Defining $\omega'$ as before, we find that the image of $\omega'$ regarded as a linear map $V_{[\mm]-\{0\}}^* \to K[G]^n$ has rank at least $n$. In other words, the flattening
$\flat_{\{0\},[\mm]-\{0\}} \omega'$ has rank at least $n>k$. This implies that $\omega'$ itself has rank larger than $k$. A similar argument applies to border rank.

For the second part, suppose $\omega$ has border rank at most $k$. Note that $\omega$ viewed as a linear map from $V_{[\mm]-\{0\}}^*$ to $V_0$ has rank at most $k$ (since this is a closed condition that is satisfied by all tensors of rank at most $k$). Then as above, one finds there are $\phi_0$, $\psi_0$ such that $\omega$ equals $(\psi_0 \otimes (\bigotimes_{i>0}
\id_{V_i})) ((\phi_0 \otimes (\bigotimes_{i>0} \id_{V_\mm}))(\omega))$; the second part follows by repeatedly applying this.
\end{proof}

\begin{re}
Note that if $\omega$ is $G$-invariant, then all $U_i$ will be $G$-stable and hence $U_i = U_i'$.

Moreover, note that we can refine Lemma~\ref{lm:AllSame} in the following way: an element $\omega$ of $V_{[\mm]}$ has (border) rank at most $k$ if and only if there are $\mm$-tuples of $G$-linear maps $\phi_i: V_i \to K[G]^k$ and $\psi_i: K[G]^k \to V_i$ such that $\psi_{[\mm]}(\phi_{[\mm]}(\omega)) = \omega$ and such that $\phi_{[\mm]}(\omega)$ has (border) rank at most $k$.

Observe that finding $\mm$-tuples of $G$-linear maps as
required (or finding that such $\mm$-tuples do not exist) is
easily done by linear algebra. In essence, this means that the problem of finding whether the (border) rank of a tensor in some tensor product exceeds $k$ be reduced to the problem of finding whether the (border) rank of a tensor in the $\mm$-fold tensor product of the space $V = K[G]^k$ exceeds $k$.
\end{re}

\begin{ex}
Consider the group $G = \ZZ/2\ZZ = \{e,g\}$ and the $8$-dimensional
$G$-module $V_0 = V_1 = K[G]^{\otimes [3]}$. Use shorthand notation
such as $[eeg]:=e \otimes e \otimes g \in V_0$. The tensor
\[ \omega:=[eee] \otimes [eee] + [ggg] \otimes [eeg]  \in V_0 \otimes V_1 \]
has rank equal to $2$. It can be regarded as a linear map from $V_1^*$
to $V_0$, and as such it has image $U_0:=\langle [eee],[ggg] \rangle$. This
subspace is already $G$-stable, so that
\[ U_0'=K[G]U_0=\langle [eee]+[ggg], [eee]-[ggg] \rangle, \]
where the two latter vectors correspond to the two different characters
of $G$. Define $\phi_0: V_0 \to K[G]^2$ by $[eee] \mapsto (e,0)$,
$[ggg] \mapsto (g,0)$ and by sending all other three-letter words
over $G$ to zero. This map is $G$-equivariant. Conversely, define
$\psi_0:K[G]^2 \to V_0$ by $\psi_0(e,0)=[eee]$, $\psi_0(g,0)=[ggg]$ and
$\psi_0(0,K[G])=\{0\}$. This $\psi_0$ is $G$-equivariant.  We
used only one copy of $K[G]$ as both characters have multiplicity one
in $U_0'$.

Next, consider $\omega$ as a linear map from $V_0^*$ to $V_1$, and let
$U_1=\langle [eee],[eeg] \rangle$ be the image of that linear map. We find
\[ U_1'=K[G]U_1=\langle [eee]+[ggg], [eee]-[ggg], [eeg]+[gge],
[eeg]-[gge] \rangle. \]
Each character has multiplicity two in $U_1'$, and we will
need the second factor $K[G]$. Define $\phi_1:V_1 \to K[G]^2$ by
\[
 [eee] \mapsto (e,0),\ [ggg] \mapsto (g,0),\ [eeg] \mapsto (0,e),\
[gge] \mapsto (0,g) \]
and by mapping all other words to zero. This map is $G$-equivariant
and surjective. Let $\psi_1:K[G]^2 \to V_1$ be the unique map such that
$\psi_1 \circ \phi_1$ restricts to the identity on $U_1'$.
Now we find that
\[ \psi_{[2]}(\phi_{[2]} \omega)
= (\psi_0 \otimes \psi_1)(\phi_0 \otimes
\phi_1)\omega=\omega \]
as stated in the lemma.
\end{ex}

Let $V$ be a $G$-representation.
Let $y_0,\ldots,y_{\di-1}$ be a basis of $V^*$.
Let $\mm \in \NN$ and denote by $\coord_{\mm}$ the
coordinate ring of the affine space $V^{\otimes [\mm]}$. Let
$\wo = (\wo_0,\ldots,\wo_{\mm-1})$ be an element of
$[\di]^{\mm}$, i.e., a word over the alphabet $[\di]$ of
length $\mm$. Then $\coord_{\mm}$ can be viewed as the
polynomial ring in the coordinates $\ten_{\wo} = \otimes_{i \in [\mm]}y_{\wo_i}$.

Several groups act naturally on $V^{\otimes [\mm]}$ in a $G$-equivariant way. First of all, denoting by $\GL_G(V)$ the group of invertible $G$-equivariant automorphisms of $V$, observe that $\GL_G(V)^{\mm}$ acts linearly on $V^{\otimes [\mm]}$ by
\[ (\phi_0,\ldots,\phi_{\mm-1})(\vect_0 \otimes \cdots \otimes \vect_{\mm-1})=
(\phi_0 \vect_0 \otimes \cdots \otimes \phi_{\mm-1} \vect_{\mm-1}), \]
and this action gives a right action on $(V^*)^{\otimes [\mm]}$ by
\[ (z_0 \otimes \cdots \otimes z_{\mm-1})(\phi_0,\ldots,\phi_{\mm-1})=
((z_0 \circ \phi_0) \otimes \cdots \otimes (z_{\mm-1} \circ \phi_{\mm-1})). \]

Second, the group $S_{\mm}$ of permutations of $[\mm]$ acts by
\[ \pi (\vect_0 \otimes \cdots \otimes \vect_{\mm-1}) =
\vect_{\pi^{-1}(0)} \otimes \cdots \otimes \vect_{\pi^{-1}(\mm-1)}. \]
This leads to the contragredient action of $S_\mm$ on the dual space
$(V^*)^{\otimes [\mm]}$ by
\[ \pi(z_0 \otimes \cdots \otimes z_{\mm-1}) =
z_{\pi^{-1}(0)} \otimes \cdots \otimes z_{\pi^{-1}(\mm-1)}. \]
Both of these extend to an action on all of $\coord_\mm$ by means of algebra automorphisms. Denote by $\G_{\mm}$ the group generated by $S_\mm$ and $\GL_G(V)^\mm$ in their representations on $V^{\otimes [\mm]}$.

Let $k \in \NN$. Given any partition of $[\mm]$ into $I,J$ we have the flattening $V^{\otimes [\mm]} \to V^{\otimes I} \otimes V^{\otimes J}$. Composing this flattening with a $(k+1) \times (k+1)$-subdeterminant of the resulting two-tensor gives a degree-$(k+1)$ polynomial in $\coord_\mm$. The linear span of all these equations for all possible partitions $I,J$ is an $\G_\mm$-submodule of $\coord_\mm$. Let $\Yp{[\mm]}$ (or more generally $\Yp{I'}$ for a finite set $I'$) denote the subvariety of $V^{\otimes [\mm]}$ (or more generally $V^{\otimes I'}$) defined by this submodule. This is an $\G_\mm$-stable variety, which will be very useful later on. Note that any contraction from $V^{\otimes [\mm]} \to V^{\otimes [\mm]-I}$ maps $\Yp{[\mm]}$ to $\Yp{[\mm]-I}$.

The following convention will be used in the remainder of this paper. Let
$\mm \in \NN$ and let $n \in [\mm]$. If $\ten \in (V^*)^{\otimes n}$, then
when we speak of the contraction from $V^{\otimes [\mm]} \to V^{\otimes
[\mm-n]}$ along $\ten$, we mean the contraction along the tensor $\ten$
viewed as an element of $(V^*)^{\otimes [\mm]-[\mm-n]}$ in the natural
way; abusing notation, we will usually denote this contraction by $\con$.
We can now state the following crucial lemma.

\begin{lem}\label{lem:ContractionPP_0}
Let $V$ be a $G$-representation. Then there an exists $\n \in \NNp$ and a
$G$-invariant tensor $\ten_0 \in (V^*)^{\otimes \n}$ such that for all
$k \in \NN$ and $m \gg k$, a tensor
$\omega\in  V^{\otimes [\mm]}$ lies in $\Yp{[\mm]}$ if (and only if)
$\con(\sigma(\omega))$ lies in $\Yp{[\mm-\n]}$
for all $\sigma \in S_\mm$ and for all $G$-equivariant contractions
$V^{\otimes [\mm]} \to V^{\otimes [\mm-\n]}$ along a tensor $\ten$ of
the form $\phi(\ten_0)$ with $\phi \in \GL_G(V)^{\n}$.
\end{lem}

In this lemma, $\mm \gg k$ means that $\mm>\M$ for some function
$\M=\M(k)$ of $k$, which we will determine below. The lemma follows from
the following lemma about contractions of subspaces of tensor powers.

\begin{lem}\label{lem:SubspaceContraction} Let $V$ be a $G$-representation
and set $n_1:=|G|$. There exists a $G$-invariant tensor $\ten_0 \in
(V^*)^{\otimes n_1}$ such that for all $k \in \NN$ and all $\mm \gg k$
and all subspaces $W \subseteq V^{\otimes [\mm]}$ the following holds:
if the dimension of $\con(\sigma(W))$ is at most $k$ for all $\sigma \in
S_\mm$ and for all tensors $\ten \in (V^*)^{\otimes n_1}$ with $\ten =
\phi(\ten_0)$ for some $\phi \in \GL_G(V)^{n_1}$, then $\dim W$ itself
is at most $k$.
\end{lem}

Again, $\mm \gg k$ means that $\mm>\M_1$ for some function $\M_1=\M_1(k)$
of $k$, which we will determine below.  To prove this lemma, we will
make use of the following combinatorial lemma concerning words over a
finite alphabet.

\begin{lem}\label{lem:combwords} Let $k,l \in \NN$ and let $A$ be a
finite alphabet. Let $w_0,\ldots,w_{k} \in A^{[l]}$ be words of length
$l$ over $A$, written down as a $[k+1] \times [l]$-array of letters
from $A$. For $\mathbf{a} \in A^{[k]}$ write
\[
J_{\mathbf{a}}:= \{j \in [l]: \forall i \in [k+1]:
(w_i)_j = a_i\} \]
for the set of positions $j$ where the array has column $\mathbf{a}$,
and for $J \subseteq [l]$ write $(w_i)_J \in A^J$ for the restriction
of the word $w_i$ to the positions in $J$. The following two
statements hold.

\begin{description}
\item[1] There exists an $\mathbf{a} \in A^{[k+1]}$ for
which $|J_{\mathbf{a}}| \geq \lceil \frac{l}{|A|^{k+1}}\rceil$.

\item[2] If $w_0,\ldots,w_{k}$ are pairwise distinct, then
there exists a subset $J \subseteq [l]$ of cardinality at most $k$ such that
$(w_0)_J,\ldots,(w_k)_J$ are pairwise distinct.
\end{description}
\end{lem}

\begin{proof} The first statement follows from immediately from
$\sum_{\mathbf{a} \in A^{[k+1]}}|J_{\mathbf{a}}| = l$.  The second
statement is proved by induction. It is clearly true for $k = 0$,
with $J=\emptyset$. Suppose $k > 0$. By induction, we may assume that
there is $J' \subseteq [l]$ of cardinality at most $k-1$ such that
$(w_0)_{J'},\ldots,(w_{k-1})_{J'}$ are pairwise distinct. In particular,
$(w_k)_{J'}$ can be equal to at most one $(w_i)_{J'}$ with $i<k$. If
it is not equal to any of these, then take $J = J'$. If it is equal to
some $(w_i)_{J'},\ i<k$, then take $j \in [l]$ such that $(w_k)_j \neq
(w_i)_j$ for this $i$ and take $J := J'\cup\{j\}$.
\end{proof}

\begin{proof}[Proof of Lemma~\ref{lem:SubspaceContraction}]
Let $\widehat{G}$ be the group of characters of $G$. Note that we have $V = \bigoplus_{\chi \in \widehat{G}}V_{\chi}$ where $V_{\chi} = \{v \in V: \forall g \in G: gv = \chi(v)v\}$. Fix a basis of $V$ of common $G$-eigenvectors, say $e_0,\ldots,e_{\di-1}$, and let $x_0,\ldots,x_{\di-1}$ be the dual basis. Such a basis exists since $V$ splits in irreducible $G$-representations of dimension $1$. Observe that each $e_i$ is an element of $V_{\chi}$ for some character $\chi$. Similarly, each $x_i$ is an element of some $V_{\chi}^*$. For each character $\chi$, let $x_{\chi} = \sum_{\{i \in
[\di]: x_i \in V_{\chi}^*\}}x_i$. Note that $x_{\chi}$ can in principle be any non-zero element of $V_{\chi}^*$, provided $V_{\chi}^* \neq \{0\}$. Indeed, we only choose a basis for technical reasons. Observe that $x_{{\chi_0}}\otimes \ldots \otimes x_{\chi_{\ms-1}}$ is $G$-invariant if the product of the corresponding characters is the trivial character.
Since $\widehat{G}$ has cardinality $|G| = n_1$, the $n_1$-fold product of any element of $\widehat{G}$ is the trivial character, and therefore $$\ten_0 := \sum_{\chi \in \widehat{G}} x_{\chi}^{\otimes n_1} \in (V^*)^{\otimes n_1}$$
is a $G$-invariant tensor.
Let $k \in \NN$. We will show that $\M_1 = k +|G|^{k+2}-|G|^{k+1}$ works for this $\ten_0$.

Let $\Gr(f,V^{\otimes [\mm]})$ denote the Grassmannian of
$f$-dimensional subspaces of
$V^{\otimes [\mm]}$, which is a projective algebraic variety over $K$. Set
\begin{align*}
Z(f,k) := \{ W \in \Gr(f,V^{\otimes [\mm]}) \mid &\dim \con(\sigma(W)) \leq k \text{
for all }\\&\ten=\phi(\ten_0), \phi \in \GL_G(V)^{n_1},\sigma
\in S_m \},\end{align*}
a closed subvariety of $\Gr(f,V^{\otimes [\mm]})$.
The assertion of the lemma is equivalent to the statement that the set of $K$-points of $Z(f,k)$ is empty if $f > k$ and $\mm > \M_1$.
So suppose the set of $K$-points of $Z(f,k)$ is nonempty for
some $f > k$, $\mm > M_1$. We will use that it is
stable under $\GL_G(V)^\mm \subseteq \G_\mm$.

Let $D \subseteq \GL_G(V)$ denote the subset of diagonal matrices with
respect to the basis $e_0,\ldots,e_{\di-1}$. Then $D^\mm$ is a connected,
solvable algebraic group and hence by Borel's Fixed Point Theorem
(\cite{Borel91}, Theorem~15.2), $D^\mm$ must have a fixed point $W$
on the projective algebraic variety $Z(f,k)$. Then also $\sigma(W)$
is a fixed point of $D^\mm$ for any $\sigma \in S_\mm$, so we can
rearrange factors if necessary. Any $D^\mm$-stable subspace is spanned
by common eigenvectors for $D^\mm$ (any algebraic representation
of $D^\mm$ is diagonalisable). Now $\omega \in V^{\otimes [\mm]}$
is a $D^\mm$-eigenvector if and only if $\omega = e_{i_0} \otimes
e_{i_1} \otimes \dots \otimes e_{i_{\mm-1}}$ (up to a nonzero scalar)
for some $i_0,\ldots,i_{\mm-1}$ with $i_j \in [\di]$ for each $j \in
[\mm]$. Say $\omega_0,\ldots,\omega_{f-1}$ form a basis of $W$ of common
$D^\mm$-eigenvectors and say $\omega_j = e_{j,0} \otimes e_{j,1} \otimes
\dots \otimes e_{j,\mm-1}$ (with each $e_{j,i}$ equal to some $e_l$). For
a contradiction, it suffices to show that there exists a tensor $\ten$ in
the $\GL_G(V)^{n_1}$-orbit of $\ten_0$ and an element $\sigma \in S_\mm$
as above such that $\con(\sigma(\omega_0)),\ldots,\con(\sigma(\omega_k))$
are linearly independent. Thus we will no longer need
$\omega_{k+1},\ldots,\omega_{f-1}$.

By the second part of Lemma~\ref{lem:combwords} there exists a subset
$J \subseteq [\mm]$ of cardinality at most $k$ such that the tensors
$\omega_{j,J} := \bigotimes_{l \in J} e_{j,l}$ for $j \in [k+1]$ are
pairwise distinct (and hence linearly independent).  Rearranging factors
we may assume that $J \subseteq [k]$. We will contract the $\omega_i$
in $n_1$ positions that all lie beyond the first $k$ positions. If those
contractions are non-zero, then they are automatically linearly independent
since their parts in the first $k$ positions are.

We now set out to find those $n_1$ positions. For each $j \in [k+1]$,
consider the word $w_j \in \widehat{G}^{[\mm]-[k]}$ of length $\mm-k$ with
letter $\chi$ at position $i$ if $e_{j,i} \in V_{\chi}$ (so we basically
consider $\omega_j$ with the first $k$ factors $e_{j,i}$ removed, and
map the remaining factors to their corresponding characters). By the
first part of Lemma~\ref{lem:combwords}, there exists a $\mathbf{\chi} =
(\chi_j)_{j\in[k+1]} \in \widehat{G}^{[k+1]}$ such that $J_{\mathbf{\chi}}
\subseteq [\mm]-[k]$ as in the lemma has cardinality at least $\lceil
\frac{\mm-k}{|\widehat{G}|^{k+1}} \rceil$. The latter expression is at
least equal to $|G|$ by choice of $\M_1$.

Now, pick a single such $\mathbf{\chi}$ and take $I \subseteq
J_{\mathbf{\chi}}$ of cardinality $|G| = n_1$ as above; by applying some
$\sigma_2$ if necessary, we may assume $I = [\mm]-[\mm-n_1]$. Note
that $I \cap [k] =\emptyset$ as promised. For each $j \in [k+1]$
and $l \in I$, we have $e_{j,l} \in V_{\chi_j}$ and we observe that
$\con_0(\omega_{j,I}) = 1$ for all $j$. One easily verifies that
$\con(\omega_j) = \omega_{j,[\mm-n_1]}$ for each $j \in [k+1]$, and these
$k+1$ tensors are linearly independent by the fact that $J \subseteq
[k]$. This concludes the proof.
\end{proof}

\begin{proof}[Proof of Lemma~\ref{lem:ContractionPP_0}]
Let $\n = n_1 = |G|$ and let $\ten_0$ be as in the proof of the previous lemma. Let $k \in \NN$ and let $\M = 2\M_1 = 2(k+|G|^{k+2}-|G|^{k+1})$. Let $\mm > \M$ and let $\omega \in V^{\otimes [\mm]}$ be an element such
that for all $\sigma \in S_\mm$, the image
of $\sigma(\omega)$ under any $G$-equivariant contraction $V^{\otimes [\mm]} \to V^{\otimes [\mm-\n]}$ along a tensor $\ten = \phi(\ten_0)$ for some $\phi \in \GL_G(V)^{\n}$ is an element of $\Yp{[\mm-\n]}$.

Let $[\mm] = I \cup J$ be any partition and consider the corresponding flattening
\[
 \flat \colon V^{\otimes [\mm]} \to V^{\otimes I} \otimes V^{\otimes J}.
\]
Replacing $\omega$ by $\sigma_1(\omega)$ for some $\sigma_1 \in S_\mm$ if necessary, we may assume $J = [\mm]-[\ms]$ and $I = [\ms]$ for some $\ms$ such that $\mm-\ms > \M_1$ without loss of generality. The statement that all $(k+1) \times
(k+1)$-subdeterminants on $\flat \omega$ are zero is equivalent to the
statement that $\flat \omega$ has rank at most $k$ when
regarded as a linear map from $(V^*)^{\otimes I}$ to $V^{\otimes J}$,
or, in other words, that the image $W \subseteq V^{\otimes [\mm]-[\ms]}$ of this
map has dimension at most $k$. Identify $V^{\otimes [\mm]-[\ms]}$ with $V^{\otimes [\mm-\ms]}$ in the natural way.

Since $\vert J \vert = \mm' > \M_1$ we may apply
Lemma~\ref{lem:SubspaceContraction} to $W$. Indeed, all contractions along tensors $\ten \in (V^*)^{\otimes \n}$ of the form $\phi(\ten_0)$ for some $\phi \in \GL_G(V)^\n$ map $\sigma'(W)$ to subspaces of
$V^{\otimes [(\mm-\ms)-\n]}$ of dimension at most $k$ for all $\sigma' \in S_{\mm-\ms}$. This follows from the
fact that this subspace is equal to the image $W'$ of the map $(V^{\otimes
I})^* \to V^{\otimes [\mm-\ms-\n]}$ obtained by first applying $\flat
\omega$ and then contracting along $\ten$. This, on the other hand,
is nothing but the map $\flat' (\omega')$ where $\omega'$ is the image of $\omega$  under the same contraction but applied to $V^{\otimes [\mm]}$, and $\flat'$ is the flattening of $[\mm-\n]$ along $[\ms],[\mm-\ms-\n]$.  Since $\omega'$ gives rise to a map of rank at most $k$ by assumption, $\dim W' \leq k$ as claimed.
Now this holds for all contractions and all factors and we may conclude
that, indeed, $\dim W \leq k$, and $\flat \omega$ has rank at most $k$.
\end{proof}

Note that in both lemmas, we do not need to compute $\con(\sigma(\omega))$ for all $\sigma \in S_\mm$; it suffices to use one $\sigma$ for each subset of $[\mm]$ of cardinality $\n$ to ensure the right factors are being contracted.

\begin{re} \label{re:Remarks}
\begin{description}
\item[1] Since $G$ is Abelian, there is a natural bijection between $\widehat{G}$ and the set of isomorphism classes of irreducible $G$-representations. For this reason, we use the letter $\chi$ both for irreducible $G$-representations and for elements of $\widehat{G}$.
\item[2] It is easily seen that the rank of $\ten_0$ as in Lemma~\ref{lem:SubspaceContraction} is bounded above by the number $N$ of distinct characters that are represented by common $G$-eigenvectors in $V^*$; in particular, the rank can generally be bounded above by $|G|$. Moreover, observe that for any $n \in \NNp$, the elements $x_{\chi}^{\otimes n}$ (with $\chi$ ranging over those characters with $V_{\chi}^* \neq \{0\}$) are linearly independent. Hence clearly, any flattening of $\ten_0$ (other than the flattenings $I = \emptyset$, $J = [|G|]$ and $I = [|G|]$, $J = \emptyset$) of the $\ten_0$ we constructed has rank equal to $N$. Therefore, $\ten_0$ has rank $N$ as well.
\item[3] Potentially, one may do better than $\n = |G|$; one may take for $\n$ the least common multiple of all orders of elements in $G$ (i.e. the exponent of $G$), and may reduce $\M$ correspondingly. For example, for the Klein $4$-group, one may take $\n = 2$ and $\M_1 = k+4^{k+1}$ instead of $\n = 4$ and $\M_1 = k+4^{k+2}-4^{k+1}$.
\item[4] If $G$ is non-trivial, then we can also take $\M_1 = |G|^{k+2}-|G|^{k+1}$ instead of $k+|G|^{k+2}-|G|^{k+1}$, or even take $\n$ to be the exponent $\exp G$ of $G$ and $\M_1 = (\n-1)|G|^{k+1}$.
\item[5] If we restrict ourselves to $G$-stable subspaces $W$ of
$V^{\otimes [\mm]}$, then instead of considering merely the dimension of
$W$, we can consider the $|G|$-tuple of multiplicities of the characters
that are represented by a common $G$-eigenvector in $W$. Using the
same $n_1$ and $\ten_0$ as in Lemma~\ref{lem:SubspaceContraction},
for each tuple $(k_{\chi})_{\chi \in \widehat{G}}$ there is an $\M_1$
such that if for all contractions as in the lemma the multiplicity
of $\chi$ in $\con(\sigma(W))$ is at most $k_{\chi}$ for each $\chi$,
then the multiplicity of $\chi$ in $W$ is at most $k_{\chi}$. In this
case, we can take $\M_1 = (\n-1)|G|^{\max_{\chi}(k_\chi)+1}$. Denoting
by $\Yp[(k_\chi)_\chi]{[\mm]}$ the set of $G$-invariant tensors $\omega$
in $V^{\otimes [\mm]}$ such that for each flattening, the multiplicity of
$\chi$ in the image of $\omega$ is at most $k_\chi$ for each $\chi$,
we can prove an analogue of Lemma~\ref{lem:ContractionPP_0} for
$\Yp[(k_\chi)_\chi]{[\mm]}$ as well. This will be particularly useful
in the case of the $G$-equivariant tree model later on.
\item[6] In general, there may be many possible choices for
$\ten_0$, (in fact, nearly all $G$-invariant tensors can be
used, as the set of tensors such that the lemma is not
satisfied is a closed set that is not equal to the set of
$G$-invariant elements of $(V^*)^{\otimes n_1}$). For
example, we could have taken $\ten_0 =
\sum_{\chi_1,\ldots,\chi_{n_1} \in \widehat{G}:\chi_1\cdot
\ldots \cdot\chi_{n_1} = 1 } \otimes_{j=1}^{n_1}x_{\chi_j} \in (V^*)^{\otimes n_1}$. In the specific case $V = K[G]$, this yields $\ten_0 = \sum_{g \in G} x_g^{\otimes n_1}$ (for some proper choice of a basis of $G$-eigenvectors of $V$), where $\{x_g\}$ is a basis dual to the basis $\{g\}$ of $K[G]$. In this case, our original choice would give $\ten_0 = \sum_{g_1,\ldots,g_n: g_1+\ldots+g_n=0}x_{g_1}\otimes\ldots \otimes x_{g_n}$.
\item[7] In Lemma~\ref{lm:AllSame}, if we restrict ourselves to
$G$-invariant tensors, then we can formulate the following refinement. Let
$(k_{\chi})_{\chi \in \widehat{G}} \in \NN^{\widehat{G}}$ and let $k =
\max_{\chi}(k_\chi)$.

    Let $\mm,n$ be natural numbers with $n \geq k+1$, and let $V_0,\ldots,V_{\mm-1}$ be $G$-representations over $K$. Let $\omega \in V_{[\mm]}$ be $G$-invariant. Then the multiplicity of $\chi$ in the image of $\flat\omega$ is at most $k_{\chi}$ for each $\chi$ and each flattening $\flat$ if and only if there are $\mm$-tuples of $G$-linear maps $\phi_i: V_i \to K[G]^k$ and $\psi_i: K[G]^k \to V_i$ such that $\psi_{[\mm]}(\phi_{[\mm]}(\omega)) = \omega$ and $\phi_{[\mm]}(\omega) \in \Yp[(k_\chi)_\chi]{[\mm]}$.
\item[8] In this lemma, we explicitly make use of the fact that $G$ is
Abelian. Indeed, if $G$ is non-Abelian, then the lemma is false. Suppose namely that $G$ is non-Abelian, and let $V$ be an irreducible $G$-representation of dimension $\di > 1$; observe that $\GL_G(V) \cong K^*$. For $n \in \NNp$, let $\ten \in (V^*)^{\otimes n}$ be a $G$-invariant tensor. Let $m \in \NNp$ with $m \geq n$ and consider the set of $S_{\mm}$-invariant tensors in $V^{\otimes [\mm]}$. This is an ${\mm+\di-1 \choose \di -1}$-dimensional subspace of $V^{\otimes [\mm]}$. The elements in this space that contract to $0$ along $\ten$ are the elements of the (non-trivial) kernel $W$ of a set of ${\mm-n+\di-1 \choose \di -1}$ linear equations. The actions of $\GL_G(V)^n$ and $S_{\mm}$ do not give any additional linearly independent equations, so we have $\phi(\ten)(\sigma(W)) = \{0\}$ for any $\phi \in \GL_G(V)^n$ and $\sigma \in S_{\mm}$, while $W\neq \{0\}$. So Lemma~\ref{lem:SubspaceContraction} does not hold in this case. Likewise, Lemma~\ref{lem:ContractionPP_0} does not hold if $G$ is non-Abelian.
\end{description}
\end{re}

\begin{ex} For $G = \ZZ/2\ZZ$, $k = 2$ and $V = K[G]^2$, the proof of the lemma combined with the remark shows we may use $\n = n_1 = 2$, $\M_1 = 16-8 = 8$ and $\M = 16$; taking basis $(e+g,0),(0,e+g),(e-g,0),(0,e-g)$ of $V$, with dual basis $x_0,x_1,x_2,x_3$ we could take $\ten_0 = (x_0+x_1)\otimes(x_0+x_1) + (x_2+x_3)\otimes(x_2+x_3)$.

In the case $V = K[G]$ and $(k_1,k_{-1}) = (1,1)$, where we write $\widehat{G} = \{1,-1\}$, we may use $\M_1 = 4$ and $\M = 8$.
\end{ex}

\section{Infinite-dimensional tensors and the flattening variety} \label{sec:Infinite}
From now on, fix $k \in \NN$, and let $V$ be a
$G$-representation. Let $\di = \dim V$ be the dimension of
$V$. For each character $\chi$ with $V_{\chi}^* \neq \{0\}$, fix $x_{\chi} \in V_{\chi}^*-\{0\}$ and let $x_{\chi} = 0$ for all other characters. Let $\n = |G|$ and define
$$\ten_0 = \sum_{\chi \in \widehat{G}} x_{\chi}^{\otimes \n}  \in (V^*)^{\otimes \n}$$
as in Section~\ref{sec:Tensors}.
For $\mm \in \NN$, we denote by $\con_0$ the contraction from $V^{\otimes
[\mm+\n]} \to V^{\otimes [\mm]}$ along the tensor $\ten_0$. More
specifically, we have
\[
\con_0(\vect_0 \otimes \ldots \otimes \vect_{\mm-1} \otimes
\vect_{\mm} \otimes \ldots \otimes \vect_{\mm+\n-1}) =
\ten_0(\vect_{\mm} \otimes \ldots \otimes \vect_{\mm+\n-1})
\cdot \vect_0 \otimes \ldots \otimes \vect_{\mm-1}. \]
Dually, this surjective map gives rise to the injective linear map
\[ (V^*)^{\otimes [\mm]} \to (V^*)^{\otimes [\mm+\n]},\
\xi \mapsto \xi \otimes \ten_0. \]
Let $\coord_\mm$ be the coordinate ring of $V^{\otimes [\mm]}$. We identify $\coord_\mm$ with the symmetric algebra $S((V^*)^{\otimes [\mm]})$ generated by the space $(V^*)^{\otimes [\mm]}$, and embed $\coord_\mm$ into $\coord_{\mm+\n}$ by means of the linear inclusion $(V^*)^{\otimes [\mm]} \to (V^*)^{\otimes [\mm+\n]}$ above.

From now on, fix $\m \in [\n]$ and define the projective limit
$$\Ainf:=\varprojlim_{\mm \in \m+\n\NN} V^{\otimes [\mm]}$$ along the
surjective linear contraction maps $\con_0$. This is, in the first place,
an uncountable-dimensional $G$-representation over $K$ (unless $\di=1$,
in which case it is one-dimensional). But it is also the dual of the
countable-dimensional {\em direct} limit of the $(V^*)^{\otimes [\mm]}$
along the inclusion maps. As a consequence, $\Ainf$ is canonically
isomorphic to the set of $K$-algebra homomorphisms $\Cinf \to K$, where
$\Cinf$ is the union $\bigcup_{\mm \in \m+\n\NN} \coord_\mm$. This gives
$\Ainf$ a Zariski topology, with closed sets given by the vanishing
of subsets of $\Cinf$. Since we are only concerned with set-theoretic
statements, we do not need to worry about points of $\Cinf$ over
$K$-algebras other than $K$; the topological space $\Ainf$
suffices for our purposes. The same applies to closed subsets
(subvarieties) of $\Ainf$ featuring below.

At a crucial step in our arguments we will use the following
more concrete description of $\Cinf$.  Extend $\ten_0$ to a basis
$\ten_0,\ten_1,\ldots,\ten_{\di^{\n}-1}$ of $(V^*)^{\otimes \n}$ of
$G$-eigenvectors. Moreover, let $y_0,\ldots,y_{\di-1}$ be any basis of
$V^*$ (not necessarily consisting of $G$-eigenvectors).  Let $\mm$ be
an element of $\m+\n\NN$. Then for any $\p \in \NN$, $(V^*)^{\otimes
[\mm+\p\n]}$ has a basis in bijection with the pairs $(\wo,w)$ with
$\wo$ a word in $[\di]^\mm$ and $w=(i_0,\ldots,i_{\p-1})$ a word of
length $p$ over
the alphabet $[\di^{\n}]$, namely,
\[ \tc_{\mm,\wo,w}:= y_{\wo_0}\otimes
y_{\wo_1}\otimes \ldots \otimes y_{\wo_\mm-1}\otimes\ten_{i_0} \otimes
\cdots \otimes \ten_{i_{\p-1}}. \]
The algebra $\coord_{\mm+p\n}$ is
the polynomial algebra in the variables $\tc_{\mm,\wo,w}$ with $w$
running over all words of length $\p$ and $\wo$ running over all words in
$[\di]^\mm$. In $\Cinf$, the coordinate $\tc_{\mm,\wo,w}$ is identified
with the variable $\tc_{\mm,\wo,w'}$ where $w'$ is obtained from $w$
by appending an infinite string of zeros at the end of $w$. If $w = 0$,
then we also write $\tc_{\mm,\wo} = \tc_{\mm,\wo,w}$.

We conclude that $\Cinf$ is a polynomial ring in countably many variables
that are (for fixed $\mm \in \m+\n\NN$) in bijective correspondence
with triples $(\mm,\wo,(i_0,i_1,\ldots))$ in which all but finitely
many $i_j$ are $0$. The finite set of positions $j$ with $i_j \neq 0$
is called the {\em support} of the word $(i_1,i_2,\ldots)$; likewise,
the set of positions $j$ with $\wo_j \neq 0$ is called the {\em support}
of $\wo$. Note that this gives a different set of variables for each
$\mm \in \m+\n\NN$; we will generally use the set of variables that is
most convenient for our purposes.

Observe that for each $\mm \in \NN$ we have natural
embeddings $\GL_G(V)^{\mm} \to \GL_G(V)^{\mm+\n}$, which
render the contraction maps $V^{\otimes [\mm+\n]} \to
V^{\otimes [\mm]}$ equivariant with respect to
$\GL_G(V)^{\mm}$. Therefore the union of $\GL_G(V)^{\mm}$ for all $\mm \in \m + \n\NN$ acts on $\Ainf$ and $\Cinf$ by passing to the limit.

Let $\Sinf$ denote the union $\bigcup_{\mm \in \m+\n\NN} S_{\mm}$, where $S_\mm$ is embedded in $S_{\mm+\n}$ as the subgroup fixing $\{\mm,\ldots,\mm+\n-1\}$.
Then $\Sinf$ is the group of all bijections $\pi: \NN \to \NN$ whose
set of fixed points has a finite complement. This group acts on $\Ainf$
and on $\Cinf$ by passing to the limit.

The action of $\Sinf$ on $\Cinf$ has the following fundamental property:
for each $f \in \Cinf$ there exists an $\mm \in \m+\n\NN$ such that whenever
$\pi,\sigma \in \Sinf$ agree on the initial segment $[\mm]$ we have
$\pi f=\sigma f$. Indeed, we may take $\mm$ equal to $\m+(\n$ times ($1$ plus the maximum of the union of the supports of words $w$ for which $\tc_{\m,\wo,w}$ appears in $f$)$)$. In this situation, there is a natural left action of the {\em increasing
monoid} $\Inc(\NN)=\{\pi: \NN \to \NN \mid \pi(0)<\pi(1)<\ldots\}$ by
means of injective algebra endomorphisms on $\Cinf$; see \cite[Section
5]{Hillar09}. The action is defined as follows: for $f \in \Cinf$,
let $\mm$ be as above. Then to define $\pi f$ for $\pi \in \Inc(\NN)$
take any $\sigma \in \Sinf$ that agrees with $\pi$ on the interval $[\mm]$
(such a $\sigma$ exists) and set $\pi f:=\sigma f$.

By construction, the $\Inc(\NN)$-orbit of any $f \in \Cinf$
is contained in the $\Sinf$-orbit of $f$. Note that the left action of
$\Inc(\NN)$ on $\Cinf$ gives rise to a {\em right} action of $\Inc(\NN)$
by means of surjective linear maps $\Ainf \to \Ainf$. A crucial argument in Section~\ref{sec:EqNoeth} uses
a map that is not equivariant with respect to $\Sinf$ but is equivariant
relative to $\Inc(\NN)$.

Recall that $\G_\mm$ is the group generated by $S_\mm$ and $\GL_G(V)^\mm$. We can now define $$\Ginf := \bigcup_{\mm \in \m+\n\NN} \G_\mm.$$ This group acts on $\Ainf$ and $\Cinf$ by passing to the limit.

Now we get back to flattenings. Recall that $\con_0: V^{\otimes [\mm+\n]}\to V^{\otimes [\mm]}$ maps $\Yp{[\mm+\n]}$ to $\Yp{[\mm]}$; this means we can define a variety $$\Yinf := \varprojlim_{\mm \in \m+\n\NN}\Yp{[\mm]} \subseteq \Ainf.$$

We describe the determinants of flattenings in more concrete terms in
the coordinates $\tc_{\mm,\wo,w}$. Let $\bwo=(\wo_0,\ldots,\wo_{k})$
be a $(k+1)$-tuple of pairwise distinct words in $[\di]^\mm$. Let
$\bwo':=(\wo'_0,\ldots,\wo'_{k})$ be another such $(k+1)$-tuple.
Suppose that the support of each $\wo_i$ is disjoint from that of each
$\wo'_j$. In this case, it makes sense to speak of $\wo_i+\wo'_j$,
which is again a word in $[\di]^\mm$. We
let $\tc[\bwo;\bwo'] = \tc_\mm[\bwo;\bwo']$ be the $(k+1) \times (k+1)$ matrix with $(i,j)$-entry equal to $\tc_{\mm,\wo_i+\wo'_j}$.
For each $\mm \in \m+\n\NN$, the variety $\Yp{[\mm]}$ is defined by the determinants of all matrices $\tc[\bwo;\bwo']$.
Then the variety $\Yinf$ is defined by the determinants of all matrices $\tc[\bwo;\bwo']$ (viewed as elements of $\Cinf$) with $\mm \in \m+\n\NN$.

Moreover, if $\bw=(w_0,\ldots,w_{k})$ and $\bw'=(w'_0,\ldots,w'_{k})$
are $k+1$-tuples of pairwise distinct infinite words with letters in
$[\di^{\n}]$ with finite support, if $\bwo$ and $\bwo'$ are as above,
and if the support of each $w_i$ is disjoint of that of each $w'_j$
for all $i,j \in [k]$, then we can define a $(k+1) \times (k+1)$-matrix
$\tc[\bwo,\bw;\bwo',\bw']$ in a way analogous to the above.

We now have the following important proposition.

\begin{prop} \label{prop:YFinite}
The flattening variety $\Yinf$ is the common zero set of finitely many $\Ginf$-orbits of $(k+1) \times (k+1)$-determinants $\det \tc[\bwo;\bwo']$ with $\bwo,\bwo'$ as above.
\end{prop}

\begin{proof}[Proof of Proposition~\ref{prop:YFinite}]
Let $\M$ be an integer such that Lemma~\ref{lem:ContractionPP_0} holds for the triple $(\M,\n,\ten_0)$.
Let $f_0,f_1,\dots,f_{N-1} \in \coord_{\ms}$ be finitely many $(k+1) \times (k+1)$-determinants that define $\Yp{[\ms]}$, where $\ms$ is the largest element of $\m+\n\NN$ that satisfies $\ms \leq \M$. Of course, in the inclusion $\coord_{\ms} \subset \Cinf$, each $f_{i}$ may be assumed to be one of the $\det \tc[\bwo;\bwo']$ for $\bwo,\bwo'$ each lists of $k+1$ words supported in $[\ms]$.

We will now show that $\omega \in \Ainf$ is an element of $\Yinf$ if and only if
$f_{i}(\g(\omega)) = 0$ for all $i$ and all $\g \in \Ginf$.  Note that $f_{i}(\g(\omega))$ is equal
to $f_{i}((\g(\omega))_{\ms})$ where $(\g(\omega))_{\ms}$ is the image of $\g(\omega)$ in $V^{\otimes [\ms]}$ under the canonical projection $\Ainf \to V^{\otimes [\ms]}$. Now if $\omega \in \Yinf$, then obviously so is $\g(\omega)$ for each $\g \in \Ginf$, and hence $(\g(\omega))_{\ms}$ is an element of $\Yp{[\ms]}$.
This shows the only if part.

For the converse, suppose that $f_{i}(\g(\omega)) = 0$ for all $i$ and all $\g \in \Ginf$. We need to show that $\omega \in \Yinf$. Equivalently, we need
to show that for all $\mm \geq \ms$ (and $\mm \in \m+\n\NN$), the image $\omega_{\mm} \in V^{\otimes [\mm]}$ of $\omega$ lies in $\Yp{[\mm]}$. Suppose $\mm = \ms+\p\n$ with $\p \in \NN$. Recall that $f_{i} \in \coord_{\ms}$ is identified in $\coord_{\ms+\p\n}$ with $f_{i}$ precomposed with the contraction $\con$ of the last $\p\n$ factors $V$ along $\ten = \ten_0^{\otimes \p}$. This means $f_i(\con((\g\omega)_\mm)) = 0$ for all $i \in [N]$ and all $\g \in \Ginf$ and hence $\con((\g\omega)_\mm) \in \Yp{[\ms]}$ for all $\g \in \Ginf$. Hence in particular, $\con((\g\omega)_\mm) \in \Yp{[\ms]}$ for all $\g \in \G_\mm$ of the form $\phi \circ \sigma$ with $\phi \in \GL_G(V)^{[\mm]-[\ms]}$ and $\sigma \in S_\mm$.

Note that for such $\g$, one has $(\g\omega)_\mm = \g(\omega)_\mm$ and moreover, the element $\con((\g\omega)_\mm)$ can be obtained by performing consecutive contractions of $\sigma(\omega_\mm)$ along tensors of the form $\phi'(\ten_0)$ (and in fact, all contractions of this form can be obtained in this way using some suitable $\g$). By repeatedly applying Lemma~\ref{lem:ContractionPP_0} this means that $\omega_{\mm} \in \Yp{[\mm]}$, and we are done.
\end{proof}

\begin{re} Again, this proof can be extended to a proof for $\Yinf[(k_\chi)_{\chi \in \widehat{G}}]$.
\end{re}

\begin{ex} \label{ex:Yinf[(1,1)]} For $G = \ZZ/2\ZZ$, $k = 2$, $\m = 0$ and $V = K[G]^3$, we have $\M = 16$, hence $\ms = 16$. Following the proof of the proposition, we find $\Yinf[2]$ is defined by the $\Ginf$-orbits of the equations that determine $\Yp[2]{[16]}$. Let $y_0,y_1,y_2,y_3,y_4,y_5 \in V^*$ be a basis dual to the basis $(e,0,0),(g,0,0),(0,e,0),(0,g,0),(0,0,g),(0,0,e)$ of $V$. For $i \in [6]$ and $I \subseteq [20]$, let $\wo_{i,I} \in [6]^{I}$ be the word of which each letter is an $i$. It is now an easy exercise to show that $\Yinf[2]$ is defined by the $\Ginf$-orbits of $\det(\tc[(\wo_{0,[n]},\wo_{2,[n]},\wo_{4,[n]});(\wo_{0,[16]-[n]},\wo_{2,[16]-[n]},\wo_{4,[16]-[n]})])$ for $n \in \{1,2,\ldots,8\}$. Here, we make use of the fact that the $\GL_G(V)$-orbit of any triple of elements $v_0,v_1,v_2 \in V$ is dense in $V$ provided that their projections to the common $G$-eigenspaces of $V$ are linearly independent as well. This holds in a somewhat larger generality as well for general $k$ and $(k+1)$-tuples of elements in $V$.

For $G = \ZZ/2\ZZ$ and $(k_1,k_{-1}) = (1,1)$, things are somewhat more
subtle. Let $V = K[G]$ and $\m= 0$. We have $\M = 8$; let $y_0,y_1 \in
V^*$ be a basis dual to the basis $e+g$, $e-g$ of $V$.  Using the proof in
\cite{Sturmfels05b} that the group-based model for $G=\ZZ/2\ZZ$ is defined
by linear and quadratic polynomials, we can show that $\Yinf[(1,1)]$
is defined by the $\Ginf$-orbits of $\tc_{8,\wo}$ where the cardinality
of $\{i \in [8]: \wo_i = 1\}$ is odd and by the $\Ginf$-orbits of
$\tc_{8,\wo_0}\tc_{8,\wo_1}-\tc_{8,\wo_2}\tc_{8,\wo_3}$ such that:

\begin{description}
\item[a] For each $i \in [8]$, the multiset $\{(\wo_0)_i, (\wo_1)_i\}$ equals the multiset $\{(\wo_2)_i, (\wo_3)_i\}$.
\item[b] For each $j \in \{1,2,3,4\}$, the cardinality of $\{i \in [8]: (\wo_j)_i = 1\}$ is even.
\end{description}

We will give some more details about this in Example~\ref{ex:Z/2Z}.

\end{ex}

\section{Equivariantly Noetherian rings and spaces} \label{sec:EqNoeth}

We briefly recall the notions of equivariantly Noetherian rings and
topological spaces, and proceed to prove the main result of this section,
namely, that $\Yinf$ is $\Ginf$-Noetherian (Theorem~\ref{thm:YEqNoether}).

If a monoid $\Pi$ has a left action by means of endomorphisms on a commutative
ring $R$ (with $1$), then we call $R$ \emph{equivariantly Noetherian}, or
$\Pi$-Noetherian, if every chain $I_0 \subseteq I_1 \subseteq \ldots$
of $\Pi$-stable ideals stabilises. This is equivalent to the statement
that every $\Pi$-stable ideal in $R$ is generated by finitely many
$\Pi$-orbits. Similarly, if $\Pi$ acts on a topological space $X$ by
means of continuous maps $X \to X$, then we call $X$ equivariantly
Noetherian, or $\Pi$-Noetherian, if every chain $X_1 \supseteq X_2
\supseteq \ldots$ of $\Pi$-stable closed subsets stabilises. If $R$ is
a $K$-algebra, then we can endow the set $X$ of $K$-valued points of $R$,
i.e., $K$-algebra homomorphisms $R \to K$ (sending $1$ to $1$), with the
Zariski topology. An endomorphism $\Rhom: R \to R$ gives a continuous map
$\pull:X \to X$ by pull-back, and if $R$ has a left $\Pi$-action making
it equivariantly Noetherian, then this induces a right $\Pi$-action on
$X$ making $X$ equivariantly Noetherian. This means, more concretely,
that any $\Pi$-stable closed subset of $X$ is defined by the vanishing of
finitely many $\Pi$-orbits of elements of $R$. If $\Pi$ happens to be a
group, then we can make the right action into a left action by taking
inverses. Here are some further easy lemmas; for their proofs we refer
to \cite{Draisma08b}.

\begin{lm}\label{lm:Closed}
If $X$ is a $\Pi$-Noetherian topological space, then any
$\Pi$-stable closed subset of $X$ is $\Pi$-Noetherian with respect to the
induced topology.
\end{lm}

\begin{lm}\label{lm:Union}
If $X$ and $Y$ are $\Pi$-Noetherian topological spaces, then the disjoint
union $X \cup Y$ is also $\Pi$-Noetherian with respect to the disjoint
union topology and the natural action of $\Pi$.
\end{lm}

\begin{lm}\label{lm:Image}
If $X$ is a $\Pi$-Noetherian topological space, $Y$ is a topological
space with $\Pi$-action (by means of continuous maps), and $\pull:X \to Y$
is a $\Pi$-equivariant continuous map, then $\im \pull$ is $\Pi$-Noetherian
with respect to the topology induced from $Y$.
\end{lm}

\begin{lm} \label{lm:Quot}
If $\Pi$ is a group and $\Pi' \subseteq \Pi$ a subgroup acting from the
left on a topological space $X'$, and if $X'$ is $\Pi'$-Noetherian,
then the orbit space $X:=(\Pi \times X')/\Pi'$ is a left-$\Pi$-Noetherian
topological space.
\end{lm}

In this lemma, $\Pi \times X'$ carries the direct-product topology
of the discrete group $\Pi$ and the topological space $X'$, the right
action of $\Pi'$ on it is by $(\pi,x)\sigma=(\pi \sigma, \sigma^{-1}
x)$, and the topology on the quotient is the coarsest topology
that makes the projection continuous. The left action of $\Pi$ on the
quotient comes from left-action of $\Pi$ on itself.  As a consequence,
closed $\Pi$-stable sets in $X$ are in one-to-one correspondence with
closed $\Pi'$-stable sets in $X'$, whence the lemma. Next we recall a
fundamental example of an equivariantly Noetherian ring, which will be
crucial in what follows.

\begin{thm}[\cite{Cohen67,Hillar09}] \label{thm:RlN}
For any Noetherian ring $Q$ and any $l \in \NN$, the ring $Q[x_{ij}
\mid i=0,\ldots,l-1;\ j=0,1,2,3,\ldots]$ is equivariantly Noetherian with
respect to the action of $\Inc(\NN)$ by $\pi x_{ij}=x_{i\pi(j)}.$
\end{thm}

Main Theorems III and IV will be derived from the following theorem,
whose proof needs the rest of this section.

\begin{thm} \label{thm:YEqNoether}
For every natural number $k$ the variety $\Yinf$ is an $\Ginf$-Noetherian
topological space.
\end{thm}

We will proceed by induction on $k$. For $k=0$ the variety
$\Yinf$ consists of a single point, the zero tensor, and the
theorem trivially holds.  Now assume that the theorem holds for
$k-1$. By Proposition~\ref{prop:YFinite} there exists $\mm \in \m+\n\NN$ and there exist $k$-tuples
$\bwo_0,\ldots,\bwo_{N-1},\bwo'_0,\ldots,\bwo'_{N-1}$ of words in $[\di]^\mm$,
such that $\tc[\bwo_a;\bwo'_a]$ is defined for all $a \in [N]$ (i.e.,
the supports of the words in $\bwo_a$ are disjoint from the supports
of the words in in $\bwo'_a$) and such that $\Yinf[k-1]$ is the common zero set of the
polynomials in $\bigcup_{a=0}^{N-1} \Ginf \det(\tc[\bwo_a;\bwo'_a])$. For each
$a \in [N]$ let $Z_a$ denote the open subset of $\Yinf$ where not
all elements of $\Ginf \det(\tc[\bwo_a;\bwo'_a])$ vanish; hence we have
\[ \Yinf = \Yinf[k-1] \cup Z_0 \cup \ldots \cup Z_{N-1}. \]
We will show that each $Z_a,\ a\in [N]$ is an $\Ginf$-Noetherian
topological space, with the topology induced from the Zariski topology on
$\Ainf$. Together with the induction hypothesis and Lemmas~\ref{lm:Union}
and \ref{lm:Image}, this then proves that $\Yinf$ is $\Ginf$-Noetherian,
as claimed.

To prove that $Z:=Z_a$ is $\Ginf$-Noetherian, consider
$\bwo:=\bwo_a=(\wo_0,\ldots,\wo_{k-1})$ and $\bwo':=\bwo'_a=(\wo'_0,\ldots,\wo'_{k-1})$ with all $\wo_i,\wo'_j \in [\di]^\mm$. Let $Z'$ denote the open subset of $\Yinf$ where
$\det(\tc[\bwo;\bwo'])$ is non-zero. This subset is stable under the group
$\Sinf'$ of all permutations $\sigma$ in $\Sinf$ that restrict to the identity
on $[\mm]$ and such that there is $\tau \in \Sinf$ such that $\sigma(\mm+\p\n+i) = \mm+\tau(\p)\n+i$ for any $\p\in \NN$ and $i \in [\n]$. Note that for such $\sigma$, one has $\sigma(\tc_{m,\wo,w}) = \tc_{m,\wo,\tau(w)}$ where $\tau(w)_{\p} = w_{\tau^{-1}(\p)}$. More explicitly, $\Sinf'$ consists of all permutations in $\Sinf$ that restrict to the identity on $[\mm]$ and that permute the set of blocks of the form $[\mm+(\p+1)\n]-[\mm+\p\n]$ with $\p \in \NN$.

\begin{lm} \label{lm:Zprime}
The open subset $Z' \subseteq \Yinf$ is an $\Sinf'$-Noetherian topological
space.
\end{lm}

\begin{proof}
We will prove that it is $\Inc(\NN)'$-Noetherian, where $\Inc(\NN)'$
is the set of all increasing maps $\pi \in \Inc(\NN)$ that restrict
to the identity on $[\mm]$ and such that there is $\tau \in \Inc(\NN)$ such that $\pi(\mm+\p\n+i) = \mm+\tau(\p)\n+i$ for any $i \in [\n]$; consult Section~\ref{sec:Infinite} for the
action of $\Inc(\NN)$. Since the $\Inc(\NN)'$-orbit of an equation is
contained in the corresponding $\Sinf'$-orbit, this will imply that $Z'$
is $\Sinf'$-Noetherian.

We start with the polynomial ring $R$ in the variables $\tc_{\mm,\wo,w}$, where
$w$ runs over all infinite words over the alphabet $[\di^{\n}]$ with the
property that the support of $w$ has cardinality at most $1$. Among these variables there are $\di^\mm$ for which
$w = 0$, namely the $\tc_{\mm,\wo}$ with $\wo \in [\di]^\mm$, and the remaining variables are
labelled by $[\di]^\mm \times ([\di^{\n}]-\{0\}) \times \NN$, where the element of $[\di^{\n}]-\{0\}$ denotes the non-zero letter of $w$ and the element of $\NN$ denotes the position at which this non-zero letter occurs. On these variables acts $\Inc(\NN)'$, fixing the first $\di^\mm$ variables and acting only on the last (position) index of the last set of variables. By
Theorem~\ref{thm:RlN} with $Q$ the ring in the first $\di^\mm$ variables
and $l=\di^\mm \times (\di^{\n}-1)$, the ring $R$ is $\Inc(\NN)'$-Noetherian.
Let $S=R[\det(\tc[\bwo;\bwo'])^{-1}]$ be the localisation of $R$ at the
determinant $\det \tc[\bwo,\bwo']$; again, $S$ is $\Inc(\NN)'$-Noetherian. We
will construct an $\Inc(\NN)'$-equivariant map $\pull$ from the set of
$K$-valued points of $S$ to $\Ainf$ whose image contains $Z'$.  We do
this, dually, by means of an $\Inc(\NN)'$-equivariant homomorphism $\Rhom$
from $\Cinf$ to $S$.

To define $\Rhom$ recursively, we first fix a partition $I,J$ of $[\mm]$ such
that the support of each $\wo_i$ is contained in $I$ and the support of each $\wo'_j$ is contained in $J$. Now if $\tc_{\mm,\wo,w} \in \Cinf$ is one of the variables in $R$,
then we set $\Rhom (\tc_{\mm,\wo,w}):=\tc_{\mm,\wo,w}$. Suppose that we have already defined $\Rhom$ on variables $\tc_{\mm,\wo,w}$ such that $\supp(w)$ has cardinality at most $b$, let $w$ be a word for which $\supp(w)$ has cardinality $b+1$ and let $\wo$ be a word in $[\di]^\mm$. We will define the image of $\tc_{\mm,\wo,w}$. Let
$\p$ be the maximum of the support of $w$, and write $w=w_{k}+w'_{k}$,
where the support of $w'_{k}$ is $\{\p\}$ and the
support of $w_{k}$ is contained in $[\p]$. Likewise, write $\wo = \wo_k+\wo'_k$ where  the support of $\wo_{k}$ is contained in $I$ and the support of $\wo'_{k}$ is contained in $J$. Consider the determinant of the matrix
\[
\tc[(\wo_0,\ldots,\wo_{k}),(w_0,\ldots,w_{k});(\wo'_0,\ldots,\wo'_{k}),(w'_0,\ldots,w'_{k})], \]
where $w_0,\ldots,w_{k-1}$ and $w'_0,\ldots,w'_{k-1}$ are
all equal to the infinite word over $[\di^\n]$ consisting of zeroes only.
This determinant equals
\[ \det(\tc[(\wo_0,\ldots,\wo_{k-1});(\wo'_0,\ldots,\wo'_{k-1})]) \cdot \tc_{\mm,\wo,w} - \pol, \]
where $\pol \in \Cinf$ is a polynomial in variables that are of the form
$\tc_{\mm,\wo_i+\wo'_j,w_i+w'_j}$ with $i,j \leq k$ but not both equal to $k$. All of these $w_i+w'_j$ have support of cardinality at most $b$ (since only $w_{k}$ and $w'_{k}$ have non-empty support and moreover, these two words have support of cardinality at most $b$), so $\Rhom(\pol)$ has already been defined. Then we set
\[ \Rhom(\tc_{\mm,\wo,w}) := \det(\tc[\bwo,\bwo'])^{-1} \Rhom(\pol). \]
The map $\Rhom$ is $\Inc(\NN)'$-equivariant by construction.

The set $Z' \subseteq \Yinf$ is contained in the image of the map
$\pull$. Indeed, this follows directly from the fact that the determinant
of the matrix
\[ \tc[(\wo_0,\ldots,\wo_{k}),(w_0,\ldots,w_{k});(\wo'_0,\ldots,\wo'_{k}),(w'_0,\ldots,w'_{k})]. \]
vanishes on $Z'$ while $\det(\tc[\bwo,\bwo'])$ does not. More precisely,
$Z'$ equals the intersection of $\Yinf$ with $\im \pull$, and hence by
Lemmas~\ref{lm:Image} and~\ref{lm:Closed} it is $\Inc(\NN)'$-Noetherian.
We already pointed out that this implies that $Z'$ is $\Sinf'$-Noetherian.
\end{proof}

Now that $Z'$ is $\Sinf'$-Noetherian, Lemma~\ref{lm:Quot} implies that
the $\Ginf$-space $(\Ginf \times Z')/\Sinf$ is $\Ginf$-Noetherian. The
map from this space to $\Ainf$ sending $(g,z')$ to $gz'$ is
$\Ginf$-equivariant and continuous, and its image is the open set $Z
\subseteq \Yinf$. Lemma~\ref{lm:Image} now implies that $Z$ is
$\Sinf$-Noetherian. We conclude that, in addition to the closed subset
$\Yinf[k-1] \subseteq \Yinf$, also the open subsets $Z_0,\ldots,Z_{N-1}$
are $\Sinf$-Noetherian. As mentioned before, this implies that
$\Yinf=\Yinf[k-1] \cup Z_0 \cup \ldots \cup Z_{N-1}$ is
$\Sinf$-Noetherian, as claimed in Theorem~\ref{thm:YEqNoether}.

\begin{re} Since $\Yinf[(k_{\chi})_\chi]$ is an $\G$-stable closed subset of $\Yinf[\sum_{\chi}k_{\chi}]$, it is an $\Ginf$-Noetherian
topological space as well.
\end{re}

\begin{re} \label{re:NonAbelian}
A natural question regarding our Main Theorems is why we
restrict to Abelian groups $G$. Do our results carry over to general
$G$, so that they apply to other phylogenetic models? Frankly,
we do not know. Certainly the fact that $G$ is Abelian is used in
the proof of Lemma~\ref{lem:SubspaceContraction}. This is used in
Proposition~\ref{prop:YFinite} to prove that $\Yinf[k]$ is defined
by finitely many polynomials up to symmetry, which in turn is used in
the induction proof in this section that $\Yinf[k]$ is Noetherian. In
the non-Abelian case, we have no idea whether (a suitable variant of)
$\Yinf[k]$ is defined by finitely many orbits of equations; and (a variant
of) $\Ainf$ seems simply too large to work with directly.  On the other
hand, in the case where $G$ has a normal Abelian subgroup that acts
transitively on $B$, finiteness results are proved in \cite{Michalek11}.
\end{re}

\section{Proofs of the main theorems} \label{sec:Proofs}
Recall that in Section~\ref{sec:Infinite}, we fixed $\n \in
\NNp$, a $G$-representation $V$, a tensor $\ten_0$ (viewed
as a contraction $\con_0: V^{\otimes[\mm+\n]} \to
V^{\otimes[\mm]}$ for each $\mm \in \NN$) and a $k \in \NN$.
Moreover, for each $\m \in [\n]$ we defined the flattening
variety $\Yinf$ which implicitly depends on all of these. In
this section, $\n$ and $\ten_0$ are still defined as before;
however, we wish to stress that some of the theorems that
follow hold for {\em any} $k \in \NN$ and any $\m \in [\n]$;
in these cases, we explicitly mention them in the statement
of the theorems. If we do not mention them, then they will be
defined implicitly as above. Finally, we will sometimes use specific
$G$-representations $V$ in our theorems.

Here are a few theorems that follow from Theorem ~\ref{thm:YEqNoether}.

\begin{thm} \label{thm:mainVII} For any fixed natural number $k$, any closed $\Ginf$-stable subset $\Zinf$ of $\Yinf$ is the common zero set in $\Ainf$ of finitely many $\Ginf$-orbits of polynomials in $\Cinf$.
\end{thm}

\begin{proof}
As $\Zinf$ is a closed $\Ginf$-stable subsets of $\Yinf$, and as $\Yinf$
is an $\Ginf$-Noetherian topological space (Theorem~\ref{thm:YEqNoether}),
$\Zinf$ is cut out from $\Yinf$ by finitely many $\Ginf$-orbits of
equations. Moreover, $\Yinf$ itself is cut out from $\Ainf$ by finitely
many $\Ginf$-orbits of Equations (Proposition~\ref{prop:YFinite}),
and hence the same is true for $\Zinf$.
\end{proof}

\begin{thm} \label{thm:mainVIII}

Let $\Zinf$ be the projective limit in $\Ainf$ of certain
$\G_\mm$-stable closed subsets $Z_\mm \subseteq \Yp{[\mm]}$
for $\mm$ running through $\m+\n\NN$ that satisfy $\con_0(Z_{\mm+\n}) \subseteq Z_\mm$ for any $\mm \in \m+\n\NN$.

Suppose moreover that there exists a tensor $\epsilon_0 \in V^{\otimes [\n]}$ such that the inclusion maps $\iota: V^{\otimes [\mm]} \to V^{\otimes [\mm + \n]},\ \omega \mapsto \omega \otimes \epsilon_0$ map $Z_\mm$ into $Z_{\mm+\n}$ and such that $\con_0 \circ \iota = \id_{V^{\otimes [\mm]}}$ (i.e. $\ten_0(\epsilon_0) = 1$).

Then there exists $\degr \in \NN$ such that for all $\mm \in \m+\n\NN$, $Z_\mm$ is defined by the vanishing of a number of polynomials of degree at most $\degr$.
\end{thm}

\begin{proof}
By Theorem~\ref{thm:mainVII} there exists a $\degr$ such that $\Zinf$ is defined in $\Ainf$ by polynomials of degree at most $\degr$; we prove that the same $\degr$ suffices in Theorem~\ref{thm:mainVIII}. Indeed, suppose that all polynomials of degree at most $\degr$ in the ideal of $Z_\mm$ vanish on a tensor $\omega \in V^{\otimes [\mm]}$. Let $\omega_\infty$ be the element of $\Ainf$ obtained from $\omega$ by successively applying $\iota$. More precisely, $\omega_\infty$ is the element in $\Ainf$ defined by $(\omega_\infty)_{\mm+\p\n} = \omega \otimes (\epsilon_0)^{\otimes \p}$ for any $\p \in \NN$. Here, $(\omega_\infty)_{\mm+\p\n}$ denotes the image of $\omega_\infty$ under the natural projection $\Ainf \to V^{\otimes [\mm+\p\n]}$.

We claim that $\omega_{\infty}$ lies in
$\Zinf$. Indeed, otherwise some $\coord_{\mm'}$ contains a polynomial $f$ of degree
at most $\degr$ that vanishes on $Z_{\mm'}$ but not on $\omega_\infty$. Now
$\mm'$ cannot be smaller than $\mm$, because then $f$ vanishes on $Z_\mm$
but not on $\omega$. But if $\mm' = \mm+\p\n \in \mm+\n\NN$, then $f \circ \iota^{\p}$ is a polynomial in $\coord_\mm$ of degree at most $\degr$ that vanishes on $Z_\mm$ but not on $\omega$. This contradicts the
assumption on $\omega$.
\end{proof}

The next theorem will be rather more subtle than the previous ones, as it involves contractions along $G$-invariant tensors that are not necessarily of length $\p\n$. For this reason, we will assume the existence of closed subsets $Z_\mm$ of $\Yp{[\mm]}$ for each $\mm \in \NN$, rather than just for each $\mm \in \m+\n\NN$.

\begin{thm} \label{thm:mainIX}

For each $\mm \in \NN$, let $Z_\mm \subseteq \Yp{[\mm]}$ be an $\G_\mm$-stable closed subset. Suppose that all contractions $V^{\otimes [\mm]} \to V^{\otimes [\ms]}$ along $G$-invariant tensors in $(V^*)^{\otimes \mm-\ms}$ map $Z_\mm$ to $Z_{\ms}$.

Suppose moreover that there exists a $G$-invariant vector $e_0 \in V$ such that the inclusion maps $\iota: V^{\otimes [\mm]} \to V^{\otimes [\mm+1]},\ \omega \mapsto \omega \otimes e_0$ map $Z_\mm$ into $Z_{\mm+1}$ for each $\mm \in \NN$ and such that $\con_0 \circ \iota^{\n} = \id_{V^{\otimes [\mm]}}$ for each $\mm \in \NN$.

Then there exists $\M \in \m+\n\NN$ such that for all $\mm \in \M + \n\NNp$ and for all $\omega \in V^{\otimes [\mm]}$ the following are equivalent:

\begin{description}
\item[1] For all $\sigma \in S_\mm$, and all contractions $\con: V^{\otimes [\mm]} \to V^{\otimes [\ms]}$ along $G$-invariant tensors in $(V^*)^{\otimes \mm-\ms}$ with $\ms \leq \M$, one has $\con(\sigma(\omega)) \in Z_{\ms}$.
\item[2] One has $\omega \in Z_\mm$.
\end{description}
\end{thm}

\begin{proof} The implication $2  \Rightarrow 1$ is trivial; we will show the implication $1 \Rightarrow 2$.
Let $\Zinf$ be the projective limit in $\Yinf$ of $Z_\mm$ for $\mm \in \m+\n\NN$.
By Theorem~\ref{thm:mainVII}, $\Zinf$ is defined (in $\Ainf$) by finitely many $\Ginf$-orbits of polynomials in $\Cinf$. This implies that there exists an $\M \in \m +\n\NN$ such that the $\Ginf$-orbits of the equations of $Z_{\M}$ define $\Zinf$. We claim that this value of $\M$ suffices for Theorem~\ref{thm:mainIX}, as well.

Indeed, suppose that $\omega \in V^{\otimes [\mm]}$ with $\mm \in \M + \n\NNp$ has the property that (for any rearrangement of its terms) all its $G$-equivariant contractions along tensors to $V^{\otimes [\ms]}$ lie in $Z_{\ms}$ and construct $\omega_\infty \in \Ainf$ as in the proof of Theorem~\ref{thm:mainVIII} (using $\iota^{\n}$ instead of $\iota$). We claim that $\omega_\infty$ lies in $\Zinf$. For this it suffices to show that for each $f$ in the ideal of $Z_{\M}$ and each $\g \in \Ginf$ the polynomial $\g f$ vanishes on $\omega_\infty$. Let $\g \in \Ginf$ and let $\mm' = \M+\p\n = \mm+\p'\n \in \mm+\n\NN$ be such that $\g \in \G_{\mm'}$. By construction, $f \in \coord_{\M}$ is identified with the function in $\coord_{\mm'}$ obtained by precomposing $f$ with the contraction
$V^{\otimes [\mm']} \to V^{\otimes [\M]}$ along the tensor $\ten_0^{\otimes
\p}$ on the last $\mm'-\M$ factors. Hence $\g f$ is the
same as contraction $V^{\otimes [\mm']} \to V^{\otimes [\M]}$ along {\em some}
$G$-invariant tensor (in some of the factors), followed by $\g' f$ for some $\g'
\in \G_{\M}$. Evaluating $\g f$ at the tensor $\omega_\infty$ is the same
as evaluating it at
\[ \omega \otimes (e_0)^{\otimes \p'\n}, \]
and boils down to contracting some, say $l$, of the factors $e_0$  and $\mm'-\M-l$ of the remaining factors $V$ along a tensor in $(V^*)^{\otimes I}$ (with $|I| = \mm'-\M$), and evaluating $\g' f$ at the result.

But this is the same thing as first applying some $\sigma \in S_{\mm}$ to $\omega$ (to ensure the right factors of $\omega$ will be contracted), then contracting $\sigma(\omega)\otimes e_0^{\otimes l} \in V^{\otimes [\mm+l]}$ to an element $\omega' \in V^{\otimes [\ms]}$ along some $G$-invariant tensor $\ten'$ in $(V^*)^{\otimes \mm+l-\ms}$ (where $\mm-\ms = |J|$) and evaluating $\g' f$ at $\sigma'(\omega' \otimes e_0^{\otimes \M-\ms})$ for some $\sigma' \in S_{\M}$. Note that $\sigma$ and $\sigma'$ are merely used to reorganise the terms of $\omega$ and $\omega' \otimes e_0^{\otimes \M-\ms}$ to avoid some cumbersome notation.

Viewing $e_0^{\otimes l}$ as a contraction from $(V^*)^{\otimes \mm+l-\ms} \to (V^*)^{\otimes \mm-\ms}$ in the natural way, we have $\tilde{\ten} :=  e_0^{\otimes l}(\ten') \in (V^*)^{\otimes \mm-\ms}$. Observe that $\omega' = \tilde{\ten}(\sigma(\omega))$ and that $\tilde{\ten}$ is $G$-invariant since both $\ten'$ and $e_0$ are $G$-invariant.

Now by assumption $\omega'$ lies in $Z_{\ms}$ (since $\ms \leq \M$), hence $\omega' \otimes e_0^{\otimes \M-\ms}$ lies in $Z_{\M}$ and hence we have $\sigma'(\omega' \otimes e_0^{\otimes \M-\ms}) \in Z_{\M}$ as well. This proves that $\g' f$ vanishes on it, so that $\g f$ vanishes on $\omega_\infty$, as claimed. Hence $\omega_\infty$ lies in $\Zinf$. But the projection $\Ainf \to V^{\otimes [\mm]}$ sends
$\omega_\infty$ to $\omega$ and $\Zinf$ to $Z_{\mm}$. Hence $\omega$ lies
in $Z_{\mm}$, as required.
\end{proof}

With these results, we can now prove our main theorems.

\begin{proof}[Proof of Main Theorem III]
By Lemma~\ref{lm:AllSame} it suffices to show that for fixed $k \in \NN$ and for $V = K[G]^n$ for some fixed $n \in \NN$ with $n > k$,
there exist $\M,\n$ such that a tensor in $V^{\otimes [\mm]},\ \mm \geq \M$, $\mm \in \m+\n\NN$
is of border rank at most $k$ as soon as all its $G$-equivariant contractions along $\mm-\ms$-tensors to $V^{\otimes [\ms]}$ have border rank at most $k$ (possibly after rearranging terms).

Recall that we defined $\ten_0$ using $x_{\chi} \in V_{\chi}^*$. Denoting the trivial character as $0$, note that $V_0^*$ is non-trivial since the sum of all basis elements of $V$ is $G$-invariant, so $x_0 \neq 0$. Moreover, $x_0$ vanishes outside of $V_0$, hence there must be an element $e_0 \in V_0$ such that $x_0(e_0) = 1$. For such $e_0$, observe that $\ten_0(e_0^{\otimes \n}) = 1$ and that $e_0$ is $G$-invariant because $V_0$ is the set of $G$-invariant elements of $V$. Now apply Theorem~\ref{thm:mainIX}.
\end{proof}

Our fourth Main Theorem requires a bit more work. We define a $G$-spaced star to be a $G$-spaced tree for which the underlying tree structure is that of a star.

\begin{lm} Let $\tre$ be a $G$-spaced star with center $\ro$ and leaves $[\mm]$. Let $I \subsetneq [\mm]$ and let $\tre'$ be the $G$-spaced star with center $\ro$ and leaves $[\mm]-I$ (and the same spaces attached to each vertex it shares with $\tre$). Let $\ten$ be a $G$-invariant tensor in $\otimes_{\ve\in I}V_i^*$. Then the map $\con: L(\tre) \to L(\tre')$ defined by $\otimes_{\ve \in [\mm]} \vect_{\ve} \mapsto \ten(\otimes_{\ve \in I}\vect_{\ve})\cdot \otimes_{\ve \in [\mm]-I}\vect_{\ve}$ maps $\eqm(\tre)$ to $\eqm(\tre')$.
\end{lm}

\begin{proof} We show that $\con(\eqmo(\tre)) \subseteq \eqmo(\tre')$. Assume without loss of generality that $I = \{\ms,\ldots,\mm-1\}$. Let $A = (A_{\ro\ve})_{\ve \sim \ro} \in \rep(\tre)$. Write $A_{\ro\ve} = \sum_{b \in B_\ro}b\otimes \vect_{b,\ve}$ for any $\ve \in \lea(\tre)$. Note that $gA_{\ro\ve} = \sum_{b \in B_\ro}(gb) \otimes (g\vect_{b,\ve}) = \sum_{b \in B_\ro}b \otimes (g\vect_{g^{-1}b,\ve})$. Since $A_{\ro\ve}$ is $G$-invariant, we find that $g^{-1}\vect_{b,\ve} = \vect_{g^{-1}b,\ve}$ for any $b \in B_\ro$, $g \in G$ and $\ve \in \lea(\tre)$.

Then we have $\con(\Psi_{\tre}(A)) = \sum_{b \in B_\ro} \ten(\otimes_{p \in I}\vect_{b,\ve}) \cdot \otimes_{\ve \in [\ms]} \vect_{b,\ve}$. Let $c_b := \ten(\otimes_{\ve \in I}\vect_{b,\ve})$. Observe that we now have $c_b = (g \ten)(\otimes_{\ve \in I}\vect_{b,\ve}) = \ten(g^{-1}\otimes_{p \in I}\vect_{b,\ve}) = \ten(\otimes_{\ve \in I}\vect_{g^{-1}b,\ve}) = c_{g^{-1}b}$ for any $g \in G$.

For $\ve \in [\ms-1]$, define $A'_{\ro\ve} = A_{\ro\ve}$ and define $A'_{\ro \ms} = \sum_{b \in B_\ro}b \otimes c_b\vect_{b,\ms}$. Observe that $A'_{\ro\ve}$ is $G$-invariant for each each $\ve \in [\ms]$, using $g^{-1}c_b\vect_{b,\ms} = c_b\vect_{g^{-1}b,\ms} = c_{g^{-1}b}\vect_{g^{-1}b,\ms}$ for any $g \in G, b \in B_\ro$. This means $A' := \rep(\tre')$. We now easily see that $\Psi_{\tre'}(A') = \con(\Psi_{\tre}(A))$, which after taking the closure concludes the proof.
\end{proof}

Suppose $V$ has a distinguished basis $B$ such that $G$ acts on $B$. It is easily seen that for a $G$-spaced star $\tre$ with center $\ro$, leaves $[\mm]$ and such that $V_{\ve} = V$ for each $\ve \in [\mm]$, one has $\eqm(\tre)$ is $\G_\mm$-stable. From now on, assume that $V$ has a distinguished basis $B$ such that $G$ acts on $B$.

Now, for $\mm \in \NN$, let $\tre_\mm$ be a $G$-spaced star with center $\ro$ with space $V_\ro$ and base $B_\ro$ of cardinality $k$, leaves $[\mm]$, and such that $V_{\ve} = V$ for each $\ve \in [\mm]$. Denote $\eqm_{\mm} = \eqm(\tre_\mm)$. Observe that $\eqm_{\mm}$ consists of tensors of rank at most $k$, hence $\eqm_{\mm} \subseteq \Yp{[\mm]}$. Fix $\m \in \NN$. We can now define $\eqminf \subseteq \Yinf \subseteq \Ainf$ as the projective limit of the $\eqm_{\mm}$ with $\mm \in \m+\n\NN$. This is the {\em infinite star model} alluded to in the introduction.

\begin{prop} \label{prop:mainIV} For any fixed space $V_\ro$ with basis $B_\ro$, the set $\eqminf$ is the common zero sets of finitely many $\Ginf$-orbits of polynomials in $\Cinf$.
\end{prop}

\begin{proof}
As $\eqminf$ is a closed $\Ginf$-stable subset of $\Yinf$ (with $k = |B_\ro|$) one can apply Theorem~\ref{thm:mainVII}.
\end{proof}

Now, we will see how we can reduce from a star with arbitrary spaces attached to the leaves to a star for which each leaf has space $V$ attached. This is the analogue of Lemma~\ref{lm:AllSame} for star models.

\begin{lem} \label{lm:LeafSame}
Let $\mm \in \NN$ and suppose $\tre$ is a $G$-spaced star with center $\ro$, with space $V_\ro$ and base $B_\ro$ of cardinality $k$, and leaves $[\mm]$, with spaces $V_\ve$ for each $\ve \in [\mm]$. Let $V = K[G]^n$ for some $n \in \NN$ with $n > k$ and let $B = \{gf_i: g \in G, f_i$ is the $i$-th standard basis vector of $K[G]^n$ viewed as a $K[G]$-module$\}$. If $\eqm_{\mm}$ is defined by polynomials of degree at most $\degr$, then so is $\eqm(\tre)$.
\end{lem}

\begin{proof} We have $\eqm(\tre)$ is contained in $L(\tre) = \bigotimes_{\ve \in [\mm]}V_\ve$ and $\eqm_{\mm}$ is contained in $L(\tre_\mm) = \bigotimes_{\ve \in [\mm]}V$. Recall that $\eqm(\tre)$ is the Zariski closure of $\eqmo(\tre)$ and $\eqm_{\mm}$ is the closure of the image of $\eqmo(\tre_\mm)$. A generic element of $\eqmo(\tre)$ is of the form $\sum_{b \in B_\ro}\otimes_{\ve \in [\mm]}\vect_{\ve,b}$ with $\sum_{b \in B_\ro}b \otimes \vect_{\ve,b}$ a $G$-invariant element of $V_\ro \otimes V_\ve$ for each leaf $\ve$. From this, we can easily conclude that any element of $\eqm(\tre)$ has border rank at most $k$. Likewise, any element of $\eqm_{\mm}$ has border rank at most $k$.

Suppose $\omega \in L(\tre) - \eqm(\tre)$. We show that there is an $\mm$-tuple of $G$-linear maps $\phi_\ve:V_\ve \to V$ such that $\phi_{[\mm]}(\omega) \not\in \eqm_{\mm}$. Note that such a $\phi_{[\mm]}$ maps $\eqm(\tre)$ to $\eqm_{\mm}$.
If this is the case, then we can immediately conclude that there is $f \in \coord_{L(\tre_\mm)}$ of degree at most $\degr$ that vanishes on $\eqm_{\mm}$ but not on $\phi_{[\mm]}(\omega)$, hence $\phi_{[\mm]}^*(f) \in \coord_{L(\tre)}$ has degree at most $\degr$, vanishes on $\eqm(\tre)$ and does not vanish on $\omega$. Hence $\eqm(\tre)$ is defined by polynomials of degree at most $\degr$.

If $\omega$ has border rank at most $k$, then by Lemma~\ref{lm:AllSame}, we can find $\mm$-tuples of $G$-linear maps $\phi_\ve:V_\ve \to V$ and $\psi_\ve: V \to V_\ve$ such that $\psi_{[\mm]}(\phi_{[\mm]}(\omega)) = \omega$. Since $\psi_{[\mm]}(\phi_{[\mm]}(\omega)) \not\in \eqm(\tre)$ by assumption (and $\psi_{[\mm]}(\eqm_{\mm}) \subseteq \eqm(\tre)$), we can conclude that $\phi_{[\mm]}(\omega) \not\in \eqm_{\mm}$.

If $\omega$ has border rank exceeding $k$, then by Lemma~\ref{lm:AllSame}, there is an $\mm$-tuple of $G$-linear maps $\phi_i:V_i \to V$ such that $\phi_{[\mm]}(\omega)$ has border rank exceeding $k$, which implies $\phi_{[\mm]}(\omega) \not\in \eqm_{\mm}$.
\end{proof}

\begin{re}

\begin{description}
\item[1] We may in fact assume $n = k$; in this case, we first test
whether some flattening of $\omega$ has rank exceeding $k$; this can
be done by equations of degree $k+1$. If not, then we can find $\mm$-tuples of $G$-linear maps $\phi_\ve:V_\ve \to V$ and $\psi_\ve: V \to V_\ve$ such that $\psi_{[\mm]}(\phi_{[\mm]}(\omega)) = \omega$ and proceed with the proof as above.
\item[2] If $V_{\ro}$ has multiplicity $k_{\chi}$ for each irreducible
representation $\chi$, then we may use $V = K[G]^{\max_{\chi}\{k_{\chi}\}}$ instead of $K[G]^n$. In fact, we may use $V = V_{\ro}$, using the fact that because of the given basis of $V$, we have $k_{\chi} = k_{\chi^{-1}}$ for each $\chi$.

    Moreover, observe that we have $\eqm_{\mm} \subseteq \Yp[(k_\chi)_\chi]{[\mm]}$.
\end{description}

\end{re}

\begin{ex} If $B = G$, then we have $K[G] \cong \bigoplus_{\chi \in
\widehat{G}}\chi$ (identifying characters and irreducible
representations in the natural way), and hence if $V_{\ro} = K[G]$,
then we have $\eqm_{\mm} \subseteq \Yp[(1)_\chi]{[\mm]}$.
\end{ex}

We now show that the (Zariski closure of the) equivariant model for a $G$-spaced star is defined in bounded degree, given a bound on the cardinality of the basis of the center of the star. After we show this, we can finally prove Main Theorem IV.

\begin{thm} \label{thm:mainV}
Let $V_\ro$ be a $G$-module with basis $B_\ro$ of cardinality $k \in \NN$. Then there exists $\degr \in \NN$ such that for each $\mm \in \NN$ and each $G$-spaced star $\tre$ with center $\ro$ with leaves $[\mm]$, one has $\eqm(\tre)$ is defined by the vanishing of a number of polynomials of degree at most $\degr$.
\end{thm}

\begin{proof}
By Lemma~\ref{lm:LeafSame} it suffices to prove that for fixed $k \in \NN$ and $V = K[G]^n$ with $n > k$, there exists a $\degr \in \NN$ such that for all $\m \in [\n]$ and for all $\mm \in\m+\n\NN$ the variety $\eqm_{\mm}$ is defined in $V^{\otimes [\mm]}$ by polynomials of degree at most $\degr$.

As in the proof of Main Theorem III, observe that there is some $G$-invariant element $e_0$ such that $\con_0(e_0^{\otimes \n}) = 1$. Let $\epsilon_0 = e_0^{\otimes \n}$.

Consider the inclusion maps
\[ \iota: V^{\otimes \mm} \to V^{\otimes \mm+\n},\ \omega \mapsto \omega
\otimes \epsilon_0. \]
Observe that $\iota(\Psi_{\tre_{\mm}}(A)) = \Psi_{\tre_{\mm+\n}}(A')$ where $A'_{\ro\ve} := A_{\ro\ve}$ if $\ve \in [\mm]$ and $A'_{\ro\ve} = (\sum_{b \in B_\ro}b)\otimes e_0$ otherwise. Moreover, each $A'_{\ro\ve}$ is $G$-invariant.

Hence this map sends $\eqm_{\mm}$ into $\eqm_{\mm+\n}$ and we easily see that it satisfies $\con_0 \circ \iota =\id_{V^{\otimes [\mm]}}$.

Thus we can apply Theorem~\ref{thm:mainVIII}.
\end{proof}

\begin{ex} \label{ex:Z/2Z} Let $B = G = \ZZ/2\ZZ$ and let $\tre$ be a $G$-spaced tree with $\mm$ leaves with space $V = K[G]$ attached to each node. Let $y_0,y_1 \in V^*$ be a basis dual to the basis $e+g$, $e-g$ of $V$.

Using the proof in \cite{Sturmfels05b} that the group-based
model for $\ZZ/2\ZZ$ is defined by linear and quadratic
polynomials, we can show that $\eqm_{\mm}$ is defined by the
$\Ginf$-orbits of $\tc_{\mm,\wo}$ where the cardinality of $\{i
\in [\mm]: \wo_i = 1\}$ is odd and by the $\Ginf$-orbits of
$\tc_{\mm,\wo_0}\tc_{\mm,\wo_1}-\tc_{\mm,\wo_2}\tc_{\mm,\wo_3}$ such that:

\begin{description}
\item[a] For each $i \in [\mm]$, the multiset $\{(\wo_0)_i, (\wo_1)_i\}$ equals the multiset $\{(\wo_2)_i, (\wo_3)_i\}$.
\item[b] For each $j \in \{1,2,3,4\}$, the cardinality of $\{i \in [\mm]: (\wo_j)_i = 1\}$ is even.
\end{description}

Note the similarity with Example~\ref{ex:Yinf[(1,1)]}. Indeed, these equations all vanish on $\Yp[(1,1)]{[\mm]}$, hence we have $\Yp[(1,1)]{[\mm]} \subseteq \eqm_{\mm}$ and therefore $\Yp[(1,1)]{[\mm]}$ is in fact equal to $\eqm_{\mm}$ in this specific case. By Example~\ref{ex:Yinf[(1,1)]}, we find that we can take $\M = 8$ in Theorem~\ref{thm:mainIX}. A more precise examination shows that we may take $\M = 5$ in this case.
\end{ex}

\begin{proof}[Proof of Main Theorem IV] Let $\tre$ be a $G$-spaced tree (over an algebraically closed field of characteristic $0$) satisfying the conditions of the theorem. By Theorem $1.7$ in \cite{Draisma07b}, one has $I(\eqm(\tre)) = \sum_{\ro \in \ver(\tre)}I(\eqm(\flat_\ro\tre))$ where $\flat_\ro\tre$ is a $G$-spaced star with center $\ro$. From this, we can easily conclude that if $\eqm(\flat_\ro\tre)$ is defined by polynomials of degree at most $\degr$ for each $\ro$, then so is $\eqm(\tre)$. Now apply Theorem~\ref{thm:mainV}.
\end{proof}

\begin{re} The proof of this theorem, along with the previous remark, shows that to describe the equations that define the equivariant model for any $G$-spaced tree, it suffices to describe the equations that define the equivariant model for any $G$-spaced star for which all nodes have the same space attached.
\end{re}

\begin{proof}[Proof of Main Theorem I] For the field $K = \CC$, by Main Theorem IV there is $\degr \in \NN$ depending on $G$ and $k = |B|$ such that $\eqm(\tre)$ is defined by polynomials of degree at most $\degr$. The tensorification of the model in the introduction is the closure of the set tensors of the form $\Psi(A)$ with $A \in \rep(\tre)$ such that $A$ satisfies an additional set of linear equalities and inequalities (certain sums must be equal to $1$ and certain coefficients must be non-negative). Since $\Psi$ is linear, these translate to linear equalities and inequalities for $\Psi(A)$. Then clearly, the closure of the set of tensors of the form $\Psi(A)$ with $A \in \rep(\tre)$ such that $A$ satisfies the linear {\em equalities} mentioned is defined by polynomials of degree at most $\max(\degr,1)$, since linear equalities can be tested by linear polynomials. The latter however equals the closure of the set of tensors of the form $\Psi(A)$ with $A \in \rep(\tre)$ such that $A$ satisfies both the linear equalities and the inequalities. Hence the tensorification of the model in the introduction is defined by polynomial equations of degree at most $\max(\degr,1)$.
\end{proof}

\begin{proof}[Proof of Main Theorem II] Let $\omega \in L(\tre)$. We
will first test whether $\omega \in \eqm(\tre)$; after that, we can
verify whether $\omega$ satisfies the additional linear equalities
mentioned in Main Theorem I. For each vertex $\ro$, view $\omega$ as
an element of $\flat_\ro(\tre)$; say $\flat_\ro(\tre)$ has leaves
$[\mm]$ and space $V_{\ve}$ for each $\ve \in [\mm]$. Use the
construction of Lemma~\ref{lm:AllSame} to produce
$\phi_{[\mm]},\psi_{[\mm]}$ such that
$\psi_{[\mm]}(\phi_{[\mm]}(\omega)) = \omega$, where $\phi_{\ve}:
V_{\ve} \to V = K[G]^{|B|+1}$. If some flattening of $\omega$ occuring
in the construction has image of rank exceeding $k = |B|$, then
conclude that $\omega \not\in \eqm(\tre)$.

Consider $\omega' = \phi_{[\mm]}(\omega)$. Take $\M$ as in Theorem~\ref{thm:mainIX}. Let $I$ be a subset of $[\mm]$ of cardinality $\p\n$ with $\mm-\p\n \leq \M$; the number of such subsets is polynomial in $\mm$ (it is $O(\mm^{\M})$).

Take a basis $\ten_1,\ldots,\ten_N$ of $G$-invariant tensors in
$(V^*)^{\otimes I}$; let $f_1$,\ldots,$f_{N'}$ be a set of polynomials
that defines $\eqm(\tre_{[\mm]-I})$. We can symbolically describe the
composition of a contraction of $\omega'$ along the formal linear
combination $\sum x_i\ten_i$ with some $f_j$ as a polynomial and test
whether this polynomial is identically $0$. If the latter is true for
all $I$ and for all flattenings, then conclude that $\omega$ lies in $\eqm(\tre)$ because of Theorem~\ref{thm:mainIX}.
\end{proof}

The set-up of our algorithm (given $M$) starting from $\tre_\mm$ is as follows.
In the deterministic setting:
\begin{description}
\item[Precomputation] Compute, once and for all, a set $E_\ms$ of equations for
$\eqm_\ms$ for all $\ms \leq M$.
\item[Input] $\omega \in V^{\otimes [\mm]}$.
\item[Output] True or false (the answer to the question whether $\omega \in \eqm_{\mm}$).
\item[Algorithm] For each $I \subseteq [\mm]$ with $|I| \geq
\mm-M$, check whether the composition of the equations in
$E_{\mm-|I|}$ with the formal contraction of $\omega$ along a general $G$-invariant element of $(V^*)^{\otimes I}$ is identically $0$. If this is the case for all $I$, then output `true', else output `false'.

\end{description}

The number of scalar arithmetic operations in this
algorithm is bounded by a polynomial in $\di^\mm$, where the
degree of that polynomial depends on the degrees of the
equations found in the pre-computation step. Observe that
running with $I$ over all sufficiently large subsets of
$[\mm]$ contributes only a factor $\mathcal{O}(\mm^M)$, which
is poly-logarithmic in $\di^\mm$.
In the probabilistic setting:

\begin{description}
\item[Precomputation] Compute, once and for all, a set $E_\ms$ of equations for
$\eqm_\ms$ for all $\ms \leq M$.
\item[Input] $\omega \in V^{\otimes [\mm]}$.
\item[Output] True or false (the (probable) answer to the question $\omega \in
\eqm_{\mm}$?).
\item[Algorithm] For each $I \subseteq [\mm]$ with $|I| \geq
\mm-M$, generate a random element $\ten$ of $(V^*)^{\otimes
I}$ and compute whether all equations in $E_{\mm-|I|}$
vanish on $\ten(\omega)$ (with $\ten$ viewed as a
contraction $V^{\otimes [\mm]} \to V^{\otimes [\mm]-I}$).
If this is the case for all $I$, then output `true', else output `false'.
\end{description}

The number of scalar arithmetic operations in this case is
linear in $\di^\mm \cdot \mm^M$.


\begin{thebibliography}{CFS11}

\bibitem[AR08]{Allman04}
Elizabeth~S. Allman and John~A. Rhodes.
\newblock Phylogenetic ideals and varieties for the general {M}arkov model.
\newblock {\em Adv. Appl. Math.}, 40(2):127--148, 2008.

\bibitem[BO11]{Bates10}
Daniel~J. Bates and Luke Oeding.
\newblock Toward a salmon conjecture.
\newblock {\em Exp. Math.}, 20(3):358--370, 2011.

\bibitem[Bor91]{Borel91}
Armand Borel.
\newblock {\em Linear Algebraic Groups}.
\newblock Springer-Verlag, New York, 1991.

\bibitem[BSS89]{Blum89}
Lenore Blum, Mike Shub, and Steve Smale.
\newblock On a theory of computation and complexity over the real numbers:
  {$NP$}-completeness, recursive functions and universal machines.
\newblock {\em Bull. Am. Math. Soc. (N.S.)}, 21(1):1--46, 07 1989.

\bibitem[BW07]{Buczynska07}
Weronika Buczy\'nska and Jaros\l aw~A. Wi\'sniewski.
\newblock On geometry of binary symmetric models of phylogenetic trees.
\newblock {\em J. Eur. Math. Soc.}, 9(3):609--635, 2007.

\bibitem[CFS08]{Casanellas07}
Marta Casanellas and Jes\'{u}s Fern\'{a}ndez-S\'{a}nchez.
\newblock The geometry of the {K}imura 3-parameter model.
\newblock {\em Advances in Applied Mathematics}, 41:265--292, 2008.

\bibitem[CFS11]{Casanellas09}
Marta Casanellas and Jes{\'u}s Fern{\'a}ndez-S{'a}nchez.
\newblock Relevant phylogenetic invariants of evolutionary models.
\newblock {\em J. Math. Pures Appl. (9)}, 96(3):207--229, 2011.

\bibitem[Cip07]{Cipra07}
Barry~A. Cipra.
\newblock Algebraic geometers see ideal approach to phylogenetics.
\newblock {\em SIAM News}, 40(6), 2007.

\bibitem[Coh67]{Cohen67}
Daniel~E. Cohen.
\newblock On the laws of a metabelian variety.
\newblock {\em J. Algebra}, 5:267--273, 1967.

\bibitem[CS05]{Casanellas05}
Marta Casanellas and Seth Sullivant.
\newblock The strand symmetric model.
\newblock In {\em Algebraic Statistics for Computational Biology}. Cambridge
  University Press, Cambridge, 2005.

\bibitem[DK09]{Draisma07b}
Jan Draisma and Jochen Kuttler.
\newblock On the ideals of equivariant tree models.
\newblock {\em Math. Ann.}, 344(3):619--644, 2009.

\bibitem[DK14]{Draisma11d}
Jan {Draisma} and Jochen {Kuttler}.
\newblock {Bounded-rank tensors are defined in bounded degree.}
\newblock {\em {Duke Math. J.}}, 163(1):35--63, 2014.

\bibitem[Dra10]{Draisma08b}
Jan Draisma.
\newblock Finiteness for the k-factor model and chirality varieties.
\newblock {\em Adv. Math.}, 223:243--256, 2010.

\bibitem[FG]{Friedland11}
Shmuel Friedland and Elizabeth Gross.
\newblock A proof of the set-theoretic version of the {S}almon {C}onjecture.
\newblock Preprint, available from \verb+http://arxiv.org/abs/1104.1776+.

\bibitem[GSS05]{Garcia05}
Luis~D. Garcia, Michael Stillman, and Bernd Sturmfels.
\newblock Algebraic geometry of {B}ayesian networks.
\newblock {\em J. Symb. Comp.}, 39(3--4):331--355, 2005.

\bibitem[HS09]{Hillar09}
Chris~J. Hillar and Seth Sullivant.
\newblock Finite {G}r\"obner bases in infinite dimensional polynomial rings and
  applications.
\newblock Preprint, available from \verb+http://arxiv.org/abs/0908.1777+, 2009.

\bibitem[LM04]{Landsberg04}
Joseph~M. Landsberg and Laurent Manivel.
\newblock On the ideals of secant varieties of {S}egre varieties.
\newblock {\em Found. Comput. Math.}, 4(4):397--422, 2004.

\bibitem[Mica]{Michalek12}
Mateusz Michalek.
\newblock On toric varieties arising from group-based models.
\newblock Preprint, available from \verb+http://arxiv.org/abs/1207.0930+.

\bibitem[Micb]{Michalek11}
Mateusz Michalek.
\newblock Toric geometry of the 3-{K}imura model for any tree.
\newblock Preprint, available from \verb+http://arxiv.org/abs/1102.4733+.

\bibitem[PS05]{Pachter05}
Lior Pachter and Bernd Sturmfels, editors.
\newblock {\em Algebraic Statistics for Computational Biology}, Cambridge,
  2005. Cambridge University Press.

\bibitem[Rai11]{Raicu10}
Claudiu Raicu.
\newblock Secant varieties of segre--veronese varieties.
\newblock 2011.
\newblock Preprint at \verb+http://arxiv.org/abs/1011.5867+.

\bibitem[RS95]{Robertson95}
Neil {Robertson} and P.D. {Seymour}.
\newblock {Graph minors. XIII: The disjoint paths problem.}
\newblock {\em {J. Comb. Theory, Ser. B}}, 63(1):65--110, 1995.

\bibitem[RS04]{Robertson04}
Neil {Robertson} and P.D. {Seymour}.
\newblock {Graph minors. XX: Wagner's conjecture.}
\newblock {\em {J. Comb. Theory, Ser. B}}, 92(2):325--357, 2004.

\bibitem[SS05]{Sturmfels05b}
Bernd Sturmfels and Seth Sullivant.
\newblock Toric ideals of phylogenetic invariants.
\newblock {\em Journal of Computational Biology}, 12:204--228, 2005.

\end{thebibliography}
\end{document}